\documentclass[12pt]{amsart}

\input xy
\xyoption{all}

\newtheorem{Fiebigtheorem}{Theorem}[section]

\newtheorem{Fiebiglemma}[Fiebigtheorem]{Lemma}
\newtheorem{Fiebigproposition}[Fiebigtheorem]{Proposition}
\newtheorem{Fiebigexamples}[Fiebigtheorem]{Examples}
\newtheorem{Fiebigdefinition}[Fiebigtheorem]{Definition}
\newtheorem{Fiebigremark}[Fiebigtheorem]{Remark}

\begin{document}

\title[Subquotient categories of  $\mathcal{O}$]{Subquotient categories of the affine category $\mathcal{O}$ at the critical level}

\author{Peter Fiebig}

 \address{Department Mathematik, FAU Erlangen--N\"urnberg, Cauerstra\ss e 11, 91058 Erlangen}
\email{fiebig@math.fau.de}

\begin{abstract}
We introduce subquotient categories of the restricted category $\mathcal{O}$ over an affine Kac--Moody algebra at the critical level and show that some of them have a realization in terms of moment graph sheaves. 
\end{abstract}
\maketitle

\section{Introduction}
 In \cite{S} Soergel introduced the subquotient category ``around the Steinberg point'' of the category of rational representations of a reductive algebraic group in positive characteristics, and related its structure to the category of modules over the associated ``algebra of coinvariants''. This subquotient category is nowadays sometimes called the {\em modular category ${\mathcal{O}}$}.  It is a quite important category, as it carries some information on the irreducible characters of the algebraic group, and it is in this category that the counterexamples to Lusztig's character formula were found by Geordie Williamson (\cite{W}), using the theory of {\em Soergel bimodules}.
 
 The restricted category ${\mathcal{O}}$ of an affine Kac--Moody algebra  at the critical level is a characteristic zero relative of the category of  modular representations (of a Lie algebra), hence it is plausible that a construction similar to Soergel's can be carried out  in this case as well. This is what we are going to establish here. We consider several subquotient categories and relate those that correspond to the support of a {\em multiplicity free} projective object that satisfies a further condition (see Section \ref{sec-ST}), to moment graph theory. Following the treatment in \cite{FieAdv} we actually work in a deformed setting and establish a localization result for all objects that admit a Verma flag. It is also convenient for us to replace the category of modules over the coinvariant algebra by the category of moment graph sheaves.

Note that a literal analogue of Soergel's result is in fact very easy to establish, i.e.\  if we consider a subquotient corresponding to the finite Weyl group, then a Harish-Chandra induction functor from the category ${\mathcal{O}}$ of the simple Lie algebra ${\mathfrak g}$ underlying the Kac--Moody algebra ${\widehat{\mathfrak g}}$ is easily shown to yield an equivalence, and a Soergel type functor in the case of ${\mathfrak g}$ can be found in \cite{FieAdv}. But the results in this paper go further (and are actually independent of the embedding  ${\mathfrak g}\subset {\widehat{\mathfrak g}}$). 

Here is a list of contents of this article. Section \ref{sec-basics} quickly introduces the basic objects and notations that we are working with. In particular, we discuss the restricted deformed category $\mathcal{O}$  and its projective objects. In Section \ref{sec-subquot} we review the construction of quotient categories of Gabriel and apply this to our situation. In Section \ref{sec-momgra} we define the relevant moment graphs, in Section \ref{sec-ST} we generalize Soergel's arguments to our situations. We prove the main Theorem, i.e.\  the equivalence that we described above,  in Section \ref{sec-mT}.

\vskip .2cm
\centerline{\bf Acknowledgements}
\vskip .1cm\noindent
The author was partially supported by the DFG grant SP1388.

\section{Basics}\label{sec-basics}
Let ${\mathfrak g}$ be a finite dimensional complex simple Lie algebra, and let ${\widehat{\mathfrak g}}$ be the associated non-twisted affine Kac--Moody algebra, i.e.\  
$$
{\widehat{\mathfrak g}}={\mathfrak g}\otimes_{\mathbb C}{\mathbb C}[t^{\pm 1}]\oplus{\mathbb C} K\oplus{\mathbb C} D
$$ 
with Lie bracket determined by 
\begin{align*}
[K,{\widehat{\mathfrak g}}] &= \{0\}, \\
[D, x\otimes t^n] &= n(x\otimes t^n), \\
[x\otimes t^n, y\otimes t^m]&=[x,y]\otimes t^{m+n}+n\delta_{m,-n} k(x,y) K
\end{align*}
for $x,y\in{\mathfrak g}$ and $m,n\in{\mathbb Z}$. Here, $k\colon{\mathfrak g}\times{\mathfrak g}\to{\mathbb C}$ denotes the Killing form and $\delta_{a,b}\in\{0,1\}$ the Kronecker delta.
With a fixed Borel subalgebra ${\mathfrak b}$ of ${\mathfrak g}$ we associate the affine Borel subalgebra
$$
{\widehat{\mathfrak b}}=({\mathfrak g}\otimes t{\mathbb C}[t]+{\mathfrak b}\otimes{\mathbb C}[t])\oplus{\mathbb C} K\oplus{\mathbb C} D,
$$
and with a fixed Cartan subalgebra ${\mathfrak h}\subset{\mathfrak b}$ we associate the affine Cartan subalgebra
$$
{\widehat{\mathfrak h}}={\mathfrak h}\oplus{\mathbb C} K\oplus{\mathbb C} D.
$$
\subsection{Affine roots} We denote by $V^\star$ the dual of a vector space $V$ and we write
$\langle\cdot,\cdot\rangle\colon V^\star\times V\to {\mathbb C}$ for the
canonical pairing. 
We denote by  $R \subset {\mathfrak h}^\star$ the set of roots of ${\mathfrak g}$ with respect to ${\mathfrak h}$, and by $R^+\subset R$ the subset of positive roots, i.e.\  the subset of  roots of ${\mathfrak b}$. The projection ${\widehat{\mathfrak h}}\to {\mathfrak h}$, $H\mapsto \overline H$, along the decomposition ${\widehat{\mathfrak h}}={\mathfrak h}\oplus {\mathbb C} K\oplus {\mathbb C} D$ allows us to embed ${\mathfrak h}^\star$ inside ${\widehat{\mathfrak h}}^\star$. If we define  $\delta\in {\widehat{\mathfrak h}^\star}$ by 
\begin{align*}
\langle\delta,{\mathfrak h}\oplus {\mathbb C} K\rangle & = \{0\}, \\ 
\langle\delta,D\rangle & = 1, 
\end{align*}
then the set $\widehat R\subset {\widehat{\mathfrak h}^\star}$ of roots of ${\widehat{\mathfrak g}}$ with respect to ${\widehat{\mathfrak h}}$ is 
$$
\widehat R=\{\alpha+n\delta \mid \alpha\in R,n\in{\mathbb Z}\}\cup\{n\delta\mid n\in {\mathbb Z}, n\ne 0\}.
$$
The set $\widehat R^{real}:=\{\alpha+n\delta\mid\alpha\in R,n\in{\mathbb Z}\}$ is called the set of affine {\em real} roots. 
Each affine real root $\beta$ has an associated affine coroot $\beta^\vee\in{\widehat{\mathfrak h}}$.
 The set $\widehat R^+$ of roots of ${\widehat{\mathfrak b}}$ with respect to ${\widehat{\mathfrak h}}$ is 
$$
\widehat R^{+}=\{\alpha+n\delta\mid \alpha\in R, n\ge 1\}\cup
 R^{+}\cup \{n\delta\mid n\ge 1\}.
$$
Let $\Pi\subset R^+$ be the set of simple roots and denote by
$\gamma\in R^+$ the highest root. Then the set of  simple affine roots is 
$$
\widehat\Pi=\Pi\cup\{-\gamma+\delta\}\subset \widehat R^+.
$$
We will use the partial order $\le$ on ${\widehat{\mathfrak h}^\star}$ that is defined by $\lambda\le\mu$ if and only if $\mu-\lambda$ can be written as a sum of elements in $\widehat R^+$.

The {\em affine Weyl group} is the subgroup $\widehat{{\mathcal W}}$  of $\operatorname{Aut}({\widehat{\mathfrak h}^\star})$ that is generated by the reflections $s_{\alpha+n\delta}\colon {\widehat{\mathfrak h}^\star}\to{\widehat{\mathfrak h}^\star}$, $\lambda\mapsto \lambda-\langle\lambda,(\alpha+n\delta)^\vee\rangle(\alpha+n\delta)$.

\subsection{Deformed representations of affine Kac--Moody algebras}
For the localization results in this article it is necessary that we work not in a ${\mathbb C}$-linear setting, but allow for more general underlying base rings. For more details we refer to \cite{FieMZ}. 
Let $S={S({\mathfrak h})}$ be the symmetric algebra of the vector space ${\mathfrak h}$. 
\begin{Fiebigdefinition} In this paper, a {\em deformation algebra} is a commutative, local, Noetherian and unital $S$-algebra.
\end{Fiebigdefinition} 
For a deformation algebra $A$ we denote by $\tau_A\colon {\widehat{\mathfrak h}}\to A$ the linear homomorphism $H\mapsto \overline H\cdot 1_A$.

\begin{Fiebigexamples} 
\begin{itemize} 
\item For any deformation algebra $A$ with maximal ideal ${\mathfrak m}$ the residue field $A/{\mathfrak m}$ with its induced $S$-algebra structure is a deformation algebra as well.
\item The most important examples of deformation algebras for us are the following: We denote by ${\widetilde S}$ the localization of $S$ at the maximal ideal $S{\mathfrak h}$, and for any prime ideal ${\mathfrak p}$ of ${\widetilde S}$ we denote by ${\widetilde S}_{\mathfrak p}$ the localization of ${\widetilde S}$ at ${\mathfrak p}$. 
\end{itemize}
\end{Fiebigexamples} 

\subsection{The deformed category ${\mathcal{O}}$}
Let $A$ be a deformation algebra and let $M$ be a ${\widehat{\mathfrak h}}$-$A$-bimodule. This means that $M$ is an $A$-module together with an $A$-linear action of ${\widehat{\mathfrak h}}$. For a weight $\lambda\in{\widehat{\mathfrak h}^\star}$ we define the $\lambda$-weight space of $M$ as 
$$
M_\lambda:=\{m\in M\mid H.m=(\lambda(H)+\tau_A(H))m\text{ for all $H\in{\widehat{\mathfrak h}}$}\}.
$$
Clearly, this is a sub-${\widehat{\mathfrak h}}$-$A$-bimodule. We say that $M$ is a {\em weight module} if $M=\bigoplus_{\lambda\in{\widehat{\mathfrak h}^\star}} M_\lambda$.


If  $M$ is  a ${\widehat{\mathfrak b}}$-$A$-bimodule, then we say that $M$ is {\em locally finite}, if every $m\in M$ lies in a sub-${\widehat{\mathfrak b}}$-$A$-bimodule that is finitely generated as an $A$-module. As $A$ is a Noetherian ring, this is equivalent to saying that $(U({\widehat{\mathfrak b}})\otimes A).m$ is a finitely generated $A$-module for any $m\in M$. Here, and in the following, we denote by $U({\mathfrak l})$ the universal enveloping algebra of a Lie algebra ${\mathfrak l}$.

\begin{Fiebigdefinition} For any deformation algebra $A$ we denote by ${\mathcal{O}}_A$ the full subcategory of the category of all ${\widehat{\mathfrak g}}$-$A$-bimodules that contains all objects $M$ that satisfy the following two properties:
\begin{itemize}
\item  as an ${\widehat{\mathfrak h}}$-$A$-bimodule, $M$ is a weight module,
\item  as an ${\widehat{\mathfrak b}}$-$A$-bimodule, $M$ is locally finite.
\end{itemize}
\end{Fiebigdefinition} 

From any $S$-linear homomorphism $A\to A^\prime$ of deformation algebras  we obtain a base change functor
$$
\cdot\otimes_A A^\prime\colon {\mathcal{O}}_A\to{\mathcal{O}}_{A^\prime}.
$$

\subsection{Homomorphisms and localizations}
Let ${\mathfrak p}\subset {\widetilde S}$ be any prime ideal. We simplify notation and write ${\mathcal{O}}_{\mathfrak p}$ instead of ${\mathcal{O}}_{{{\widetilde S}}_{\mathfrak p}}$. For any object $M$ in ${\mathcal{O}}_{{\widetilde S}}$ we denote by $M_{\mathfrak p}:=M\otimes_{{\widetilde S}} {{\widetilde S}}_{\mathfrak p}$ the object in ${\mathcal{O}}_{{\mathfrak p}}$ obtained by base change. 
Now, if $M$ and $N$ are objects in ${\mathcal{O}}_{{\widetilde S}}$ that are torsion free as ${\widetilde S}$-modules, then the natural homomorphism $\operatorname{Hom}_{{\mathcal{O}}_{{\widetilde S}}}(M,N)\to\operatorname{Hom}_{{\mathcal{O}}_{\mathfrak p}}(M_{\mathfrak p},N_{\mathfrak p})$ given by the base change functor ${\mathcal{O}}_{{\widetilde S}}\to{\mathcal{O}}_{{\mathfrak p}}$  is injective. In particular, we can view both $\operatorname{Hom}$-spaces as subspaces in $\operatorname{Hom}_{{\mathcal{O}}_{(0)}}(M_{(0)},N_{(0)})$. 

We denote by ${\mathfrak P}$ the set of prime ideals of ${\widetilde S}$ of height one. Note that ${\widetilde S}=\bigcap_{{\mathfrak p}\in{\mathfrak P}}  {{\widetilde S}}_{\mathfrak p}$, where the intersection is taken inside ${\operatorname{Quot}}\, {\widetilde S}={{\widetilde S}}_{(0)}$. Recall that we say that an ${\widetilde S}$-module is {\em reflexive} if it is the intersection of all its localizations at prime ideals of height one. The following result makes our theory work.

\begin{Fiebiglemma} \label{lemma-bcone} Suppose that $M,N$ are objects in ${\mathcal{O}}_{{\widetilde S}}$ and that $M$ is torsion free as an ${\widetilde S}$-module and $N$ is reflexive as an ${\widetilde S}$-module.  Then the natural homomorphism
$$
\operatorname{Hom}_{{\mathcal{O}}_{{\widetilde S}}}(M,N)\to\bigcap_{{\mathfrak p}\in{\mathfrak P}}\operatorname{Hom}_{{\mathcal{O}}_{{\mathfrak p}}}(M_{\mathfrak p},N_{\mathfrak p})
$$
is an isomorphism. 
\end{Fiebiglemma} 
\begin{proof} This is easy to prove.  If $f\colon M_{(0)}\to N_{(0)}$ is such that $f(M_{\mathfrak p})\subset N_{\mathfrak p}$, then, as $N$ is reflexive, we have $f(M)\subset \bigcap f(M_{\mathfrak p})\subset \bigcap N_{\mathfrak p}=N$, hence $f$ induces a morphism from $M$ to $N$.
\end{proof}

\subsection{Deformed Verma modules and simple quotients}

For any weight $\lambda\in{\widehat{\mathfrak h}^\star}$ we denote by $A_\lambda$ the $A$-${\widehat{\mathfrak b}}$-bimodule that is free of rank $1$ as an $A$-module and on which ${\widehat{\mathfrak b}}$ acts via the character ${\widehat{\mathfrak b}}\to{\widehat{\mathfrak h}}\stackrel{\tau_A+\lambda}\to A$ (here the homomorphism ${\widehat{\mathfrak b}}\to{\widehat{\mathfrak h}}$ is a splitting of the inclusion). Then 
$$
\Delta_A(\lambda):=U({\widehat{\mathfrak g}})\otimes_{U({\widehat{\mathfrak b}})} A_\lambda
$$ 
is a ${\widehat{\mathfrak g}}$-$A$-bimodule and as such it is an object in ${\mathcal{O}}_A$. It has a unique simple quotient  $L_A(\lambda)$, and each simple object in ${\mathcal{O}}_A$ is isomorphic to $L_A(\lambda)$ for a unique $\lambda\in{\widehat{\mathfrak h}^\star}$ (cf.\  \cite{FieMZ}).

\subsection{Decomposition of ${\mathcal{O}}_A$ according to the central character}
For any $c\in{\mathbb C}$ set 
$$
{{\widehat{\mathfrak h}}^\star}_c:=\{\lambda\in {\widehat{\mathfrak h}^\star}\mid \langle\lambda,K\rangle=c\}.
$$
Let $M$ be a ${\widehat{\mathfrak g}}$-$A$-bimodule that is a weight module. Then the subspace $M_c:=\bigoplus_{\lambda\in{\widehat{\mathfrak h}^\star}_c}M_\lambda$ is a sub-${\widehat{\mathfrak g}}$-$A$-bimodule of $M$ and we have 
$$
M=\bigoplus_{c\in{\mathbb C}} M_c.
$$
We say that $M$ {\em is of level $c$} if $M=M_c$. We denote by ${\mathcal{O}}_{A,c}$ the full subcategory  of ${\mathcal{O}}_A$ that contains all objects of level $c$. 

Let $\rho\in{\widehat{\mathfrak h}^\star}$ be any element with the property that $\rho(\alpha^\vee)=1$ for each simple affine root $\alpha\in\widehat\Pi$. Even though $\rho$ is not uniquely determined, the number ${\operatorname{crit}}:=-\rho(K)\in{\mathbb C}$ is. It is called the {\em critical level}, and the category ${\mathcal{O}}_{A,{\operatorname{crit}}}$ is called the deformed category ${\mathcal{O}}$ {\em at the critical level}. In this article we focus on the structure of ${\mathcal{O}}_{{\widetilde S},{\operatorname{crit}}}$. There is a huge center acting on this category.

 \subsection{The restricted deformed category ${\mathcal{O}}$ at the critical level}
 
 For the following we refer to Section 5 in \cite{AF12} and the references mentioned there.
Let $V^{{\operatorname{crit}}}({\mathfrak g})$ be the vertex algebra associated with ${\mathfrak g}$ at the critical level. It contains a huge center ${\mathfrak z}({\mathfrak g})$. Moreover, the vertex algebra, and hence its center, acts naturally on any object in ${\mathcal{O}}_{A,{\operatorname{crit}}}$ by ${\mathfrak g}\otimes{\mathbb C}[t^{\pm 1}]\oplus {\mathbb C} K$-module homomorphisms, but the action does not commute with the action of $D$ in general. In fact, the commutator with $D$ yields a grading ${\mathfrak z}({\mathfrak g})=\bigoplus_{n\in{\mathbb Z}}{\mathfrak z}({\mathfrak g})_n$. 

\begin{Fiebigdefinition} We say that $M\in{\mathcal{O}}_{A,{\operatorname{crit}}}$ is {\em restricted} if all $z\in{\mathfrak z}({\mathfrak g})_n$ with $n\ne 0$ annihilate $M$. We denote by ${\overline{\mathcal O}} _{A,{\operatorname{crit}}}$ the full subcategory of ${\mathcal{O}}_{A,{\operatorname{crit}}}$ that contains all restricted objects.
\end{Fiebigdefinition}

For an object $M$ of ${\mathcal{O}}_{A,{\operatorname{crit}}}$ we denote by $M^\prime\subset M$ the sub-${\widehat{\mathfrak g}}$-$A$-bimodule that is generated by $z.M\subset M$ for all $z\in{\mathfrak z}({\mathfrak g})_n$ with $n\ne 0$. Then $\overline M:=M/M^\prime$ is the largest quotient of $M$ that is contained in ${\overline{\mathcal O}} _{A,{\operatorname{crit}}}$, and $M\mapsto \overline M$ defines a functor ${\mathcal{O}}_{A,{\operatorname{crit}}}\to{\overline{\mathcal O}} _{A,{\operatorname{crit}}}$ that is left adjoint to the inclusion functor.  In particular, we denote by
$$
{\overline{\Delta}} _A(\lambda):=\overline{\Delta_A(\lambda)}
$$
the restricted deformed Verma module with highest weight $\lambda$. Note that $\overline{L_A(\lambda)}=L_A(\lambda)$. Again we simplify notation and write ${\overline{\Delta}}_{\mathfrak p}(\lambda)$ instead of ${\overline{\Delta}}_{\widetilde S_{\mathfrak p}}(\lambda)$.

\subsection{The restricted block decomposition}\label{sec-blockdec}   If $M$ is an object in ${\mathcal{O}}_A$ and $\lambda\in{\widehat{\mathfrak h}^\star}$, we write $[M:L_A(\lambda)]\ne 0$ whenever there is a subobject $N$ of $M$ which admits an epimorphism $N\to L_A(\lambda)$. Let us denote by ${\mathbb K} =A/{\mathfrak m}$ the residue field of $A$. 
We denote by $\sim_A$ the equivalence relation on ${\widehat{\mathfrak h}^\star}_{{\operatorname{crit}}}$ that is generated by $\lambda\sim_A\mu$ if $[{\overline{\Delta}} _{A} (\lambda):L_{A} (\mu)]\ne 0$. Note that  we have $[{\overline{\Delta}} _{A} (\lambda):L_{A} (\mu)]\ne 0$ if and only if $[{\overline{\Delta}} _{{\mathbb K} } (\lambda):L_{{\mathbb K} } (\mu)]\ne 0$. 

For a union $\Omega\subset{\widehat{\mathfrak h}^\star}_{{\operatorname{crit}}}/\sim_A$ of equivalence classes we define ${\overline{\mathcal O}} _{A,\Omega}\subset{\overline{\mathcal O}} _{A,{\operatorname{crit}}}$ as the full subcategory that contains all objects $M$ with the property that $[M:L_A(\lambda)]\ne 0$ implies $\lambda\in\Omega$. Note that we have ${\overline{\Delta}} _A(\lambda)\in{\overline{\mathcal O}} _{A,\Omega}$ if and only if $\lambda\in\Omega$. The following is the restricted block decomposition:
\begin{Fiebigtheorem}  [\cite{AFLin}] The functor
\begin{align*}
\prod_{\Lambda\in{\widehat{\mathfrak h}^\star}_{{\operatorname{crit}}}/\sim_A}{\overline{\mathcal O}} _{A,\Lambda}&\to{\overline{\mathcal O}} _{A,{\operatorname{crit}}},\\
(M_\Lambda)&\mapsto \bigoplus_\Lambda M_\Lambda,
\end{align*}
is an equivalence of categories. 
\end{Fiebigtheorem} 
\subsection{The restricted linkage principle}
The equivalence classes with respect to $\sim_A$ can be described quite explicitely: For  $\Lambda\in{\widehat{\mathfrak h}^\star}_{\operatorname{crit}}/\sim_A$ and $\lambda\in\Lambda$ we define
$$
\widehat R_\Lambda:=\{\beta\in\widehat R^{real}\mid\langle\lambda+\tau_A+\rho,\beta\rangle\in{\mathbb Z}\}.
$$
This is independent of the choice of $\lambda$ in $\Lambda$ and is called the set of {\em integral (real) roots} with respect to $\Lambda$.  The {\em integral Weyl group} associated with $\Lambda$ is the subgroup $\widehat{{\mathcal W}}_\Lambda$ of $\widehat{{\mathcal W}}$ that is generated by all reflections $s_{\beta}$ with $\beta\in\widehat R_\Lambda$. In \cite{AFLin} it is shown that 
$$
\Lambda=\widehat{{\mathcal W}}_\Lambda.\lambda,
$$
where on the right hand side $\widehat{{\mathcal W}}_\Lambda$ acts on ${\widehat{\mathfrak h}^\star}$ via the {\em dot-action}, i.e.\  via $w.\lambda=w(\lambda+\rho)-\rho$. Note that, as before,  this definition does not depend on the choice of $\rho$. 

\subsection{Objects admitting a restricted Verma flag}
Our localization result holds for the class of the following objects. Let us fix a $\sim_A$-equivalence class $\Lambda$ in ${\widehat{\mathfrak h}^\star}_{{\operatorname{crit}}}$.
\begin{Fiebigdefinition} We say that $M\in{\overline{\mathcal O}} _{A,\Lambda}$ {\em admits a restricted Verma flag} if there is a finite filtration $0=M_0\subset M_1\subset\dots\subset M_n=M$ such that $M_i/M_{i-1}$ is isomorphic to a restricted deformed Verma module for all $i=1,\dots, n$. 
\end{Fiebigdefinition}
We denote by ${\overline{\mathcal O}} _{A,\Lambda}^{V}$ the full subcategory of ${\overline{\mathcal O}} _{A,\Lambda}$ that contains all objects that admit a restricted Verma flag. 
Suppose that $M$ admits a restricted Verma flag and let $0=M_0\subset M_1\subset \dots\subset M_n=M$ be a filtration as in the definition above. For each $\mu\in\Lambda$ we define
$$
(M,{\overline{\Delta}} _A(\mu)):=\#\{i\in\{1,\dots,n\}\mid M_i/M_{i-1}\cong {\overline{\Delta}} _A(\mu)\}.
$$
This number is independent of the filtration and is called the {\em Verma multiplicity} of $M$ at $\mu$. The {\em ${\overline{\Delta}}$-support} of $M$ we define  by  
$$
\operatorname{supp}_{{\overline{\Delta}} } M:=\{\mu\in\Lambda\mid (M:{\overline{\Delta}} _A(\mu))\ne 0\}.
$$

\subsection{Subobjects in restricted Verma modules}

 Let $\lambda,\mu\in\Lambda$ with $\lambda\le\mu$. 
\begin{Fiebigdefinition} We say that $\lambda$ is {\em linked} to $\mu$ if there is a sequence $\beta_1,\dots, \beta_n\in\widehat R^+\cap\widehat R_\Lambda$ such that 
$$
\lambda<s_{\beta_1}.\lambda<s_{\beta_2}s_{\beta_1}.\lambda<\cdots<s_{\beta_n}\cdots s_{\beta_1}.\lambda=\mu.
$$
\end{Fiebigdefinition}
Note that if $A=\mathbb K$ is a field, then for objects  $M$ in $\mathcal{O}_{\mathbb K}$ with finite dimensional weight spaces one can define the Jordan--H\"older multiplicity $[M:L_{\mathbb K}(\mu)]\in{\mathbb N}$ even though there might not be a Jordan--H\"older series of finite length. 
For linked weights we have the following result:

\begin{Fiebiglemma} \label{lemma-antidom} Suppose that $A={\mathbb K}$ is a field. Suppose that  $\lambda,\chi\in\Lambda$ satsify the following:
\begin{enumerate}
\item[a)] $\lambda$ is linked to $\chi$ and  $[{\overline{\Delta}} _{\mathbb K}(\chi),L_{\mathbb K}(\lambda)]=1$,
\item[b)] $\lambda+\delta\not\le\chi$.
\end{enumerate}
  Then the following holds:
\begin{enumerate}
\item There is a unique submodule $M$ of ${\overline{\Delta}} _{\mathbb K}(\chi)$ of highest weight $\lambda$. 
\item Let $N\subset {\overline{\Delta}} _{\mathbb K}(\chi)$ be a submodule of highest weight $\nu>\lambda$ and suppose that $\nu$ is linked to $\lambda$.   Then $N$ contains $M$. 
\end{enumerate}
\end{Fiebiglemma} 
\begin{proof} We first prove statement (1). The multiplicity assumption implies that there is at most one submodule of highest weight $\lambda$ in ${\overline{\Delta}} _{\mathbb K}(\chi)$.  By a Theorem of Neidhardt  there exists a non-zero homomorphism $f\colon{{\Delta}} _{\mathbb K}(\lambda)\to{{\Delta}} _{\mathbb K}(\chi)$ between the non-restricted Verma modules (see \cite{N} and \cite[Section 2.11]{MP}). As $\lambda+\delta$ is not smaller or equal to $\chi$, it  is not a weight of $\Delta_{\mathbb K}(\chi)$, and hence the restricted homomorphism $\overline f\colon{\overline{\Delta}} _{\mathbb K}(\lambda)\to{\overline{\Delta}} _{\mathbb K}(\chi)$ is non-zero as well, and so its image is of highest weight $\lambda$. 

Now we prove statement (2). Note that we can assume that $N$ is a highest weight module with highest weight $\nu$. Hence there is a non-zero  homomorphism ${\overline{\Delta}} _{\mathbb K}(\nu)\to N$, and the composition ${\overline{\Delta}} _{\mathbb K}(\lambda)\to{\overline{\Delta}} _{\mathbb K}(\nu)\to N\to M$ is the homomorphism we considered above.
\end{proof}

\subsection{A topology on $\Lambda$}

We need the following notions.

\begin{Fiebigdefinition} \begin{itemize}
\item A subset ${\mathcal J}$ of $\Lambda$ is called {\em open}, if for all $\lambda,\mu\in\Lambda$ with $\mu\in{\mathcal J}$ and $\lambda\le \mu$ we have $\lambda\in{\mathcal J}$. 
\item A subset ${\mathcal I}$ of $\Lambda$ is called {\em  closed} if $\Lambda\setminus{\mathcal I}$ is open. 
\item A subset ${\mathcal K}$ of $\Lambda$ is called {\em locally closed} if there is an open subset ${\mathcal J}$ and a closed subset ${\mathcal I}$ such that ${\mathcal K}={\mathcal J}\cap{\mathcal I}$.
\item A subset ${\mathcal T}$ of $\Lambda$ is called {\em locally bounded} if for any $\lambda\in{\mathcal T}$ the set $\{\mu\in{\mathcal T}\mid \lambda\le\mu\}$ is finite.
\end{itemize}
\end{Fiebigdefinition}
Note that this indeed defines a topology on the set $\Lambda$. This topology has the property that arbitrary intersections of open subsets are open, and arbitrary unions of closed subsets are closed.  

We will use the shorthand $\{\le\lambda\}$ for the set $\{\mu\in\Lambda\mid\mu\le\lambda\}$ and we define $\{<\lambda\},\{\ge\lambda\},\dots$ in an analogous fashion.

\begin{Fiebigremark} Suppose that ${\mathcal K}\subset\Lambda$ is locally closed. Then ${\mathcal K}_+:=\bigcup_{\lambda\in{\mathcal K}} \{\le \lambda\}$ is the smallest open subset containing ${\mathcal K}$, and ${\mathcal K}_-:=\bigcup_{\mu\in{\mathcal K}_+\setminus{\mathcal K}}\{\le \mu\}$ is the (open) complement of ${\mathcal K}$ in ${\mathcal K}_+$.
\end{Fiebigremark}

 \subsection{Projectives in truncated subcategories}

Let  $M$ be an object in ${\overline{\mathcal O}} _{A,\Lambda}$. For any closed subset ${\mathcal I}$ of $\Lambda$ we define $M_{\mathcal I}\subset M$ as the sub-${\widehat{\mathfrak g}}$-$A$-bimodule generated by the subspace $\bigoplus_{\nu\in{\mathcal I}} M_\nu$ in $M$. If ${\mathcal J}$ is the open complement of ${\mathcal I}$ we set
$M^{\mathcal J}=M/M_{\mathcal I}$.

\begin{Fiebigdefinition} Let ${\mathcal J}$ be an open subset of $\Lambda$. We denote by ${\overline{\mathcal O}} _{A,\Lambda}^{{\mathcal J}}$ the full subcategory of ${\overline{\mathcal O}} _{A,\Lambda}$ that contains all objects $M$ with the property $M=M^{\mathcal J}$ (i.e.\  $M_\nu=0$ for all $\nu\in{\mathcal I}$).
\end{Fiebigdefinition}

These are the {\em  truncated subcategories}. Here,  projective objects exist:
 \begin{Fiebigtheorem} [\cite{AFLin,FProj}] Let ${\mathcal J}$ be a locally bounded open subset of $\Lambda$. For any $\lambda\in{\mathcal J}$ there exists an up to isomorphism unique object ${\overline{P}}^{{\mathcal J}}_A(\lambda)$ with the following properties:
 \begin{enumerate}
 \item ${\overline{P}}_A^{\mathcal J}(\lambda)$ is projective and indecomposable in ${\overline{\mathcal O}} _{A,\Lambda}^{{\mathcal J}}$.
 \item There is a surjective homomorphism ${\overline{P}}_A^{\mathcal J}(\lambda)\to{\overline{\Delta}} _A(\lambda)$.
 \end{enumerate}
 \end{Fiebigtheorem} 
 Here is a list of properties of the projectives:
 \begin{Fiebigtheorem} [\cite{AFLin,FProj}] \label{thm-propproj} 
 \begin{enumerate}
 \item Let ${\mathcal J}^\prime\subset{\mathcal J}$ be open subsets of $\Lambda$. Then 
 $$
 ({\overline{P}}^{{\mathcal J}}_A(\lambda))^{{\mathcal J}^\prime}\cong {\overline{P}}_A^{{\mathcal J}^\prime}(\lambda).
 $$
 \item For a homomorphism $A\to A^\prime$ of deformation algebras, the object ${\overline{P}}^{\mathcal J}_A(\lambda)\otimes_A A^\prime$ is projective in ${\overline{\mathcal O}} _{A^\prime,\Lambda}^{\mathcal J}$ and it maps surjectively onto ${\overline{\Delta}} _{A^\prime}(\lambda)$.
 \item For a homomorphism $A\to A^\prime$ of deformation algebras and a projective object $P$ in ${\overline{\mathcal O}} _{A,\Lambda}^{{\mathcal J}}$, we have a functorial isomorphism
 $$
 \operatorname{Hom}_{{\overline{\mathcal O}} _{A,\Lambda}}(P,M)\otimes_AA^\prime\stackrel{\sim}\to\operatorname{Hom}_{{\overline{\mathcal O}} _{A^\prime,\Lambda}}(P\otimes_AA^\prime, M\otimes_AA^\prime).
 $$
 \item The object ${\overline{P}}^{{\mathcal J}}_A(\lambda)$ admits a restricted deformed Verma flag and for the multiplicities we have
 $$
 ({\overline{P}}^{{\mathcal J}}_A(\lambda),{\overline{\Delta}} _A(\mu))= \begin{cases}
 0,&\text{ if $\mu\not\in{\mathcal J}$,}\\
 [{\overline{\Delta}} _{\mathbb K} (\mu):L_{\mathbb K} (\lambda)],&\text{ if $\mu\in{\mathcal J}$.}
 \end{cases}
  $$
 \end{enumerate} 
 \end{Fiebigtheorem} 
 Part (4) in the theorem above is called the {\em BGG-reciprocity}. Again we simplify notation and write ${\overline{P}}^{{\mathcal J}}_{\mathfrak p}(\lambda)$ instead of ${\overline{P}}^{{\mathcal J}}_{\widetilde S_{\mathfrak p}}(\lambda)$.

\subsection{The structure of generic and subgeneric projectives}
Suppose that $\alpha\in\widehat R_\Lambda$. Then for any $\lambda\in\Lambda$ we define $\alpha\uparrow\lambda$ to be the minimal element in $\{s_{\alpha+n\delta}.\lambda\mid n\in{\mathbb Z}, s_{\alpha+n\delta}.\lambda\ge\lambda\}$. 

\begin{Fiebigproposition}[\cite{Fsub}]\label{prop-strucproj} Let ${\mathfrak p}$ be a prime ideal of ${\widetilde S}$, let ${\mathcal J}$ be an open and locally bounded subset of $\Lambda$ and let $\lambda$ be an element in ${\mathcal J}$.
\begin{enumerate}
\item If $\alpha^\vee\not\in{\mathfrak p}$ for all $\alpha\in R$, then ${\overline{P}}_{\mathfrak p}^{\mathcal J}(\lambda)\cong{\overline{\Delta}} _{\mathfrak p}(\lambda)$.
\item Suppose that  ${\mathfrak p}=\alpha^\vee {\widetilde S}$ for some $\alpha\in R$, but $\langle\lambda,\alpha^\vee\rangle\not\in{\mathbb Z}$ or $\alpha\uparrow\lambda\not\in{\mathcal J}$. Then ${\overline{P}}_{\mathfrak p}^{\mathcal J}(\lambda)\cong{\overline{\Delta}} _{\mathfrak p}(\lambda)$.
\item   Suppose that  ${\mathfrak p}=\alpha^\vee {\widetilde S}$ for some $\alpha\in R$ and $\langle\lambda,\alpha^\vee\rangle\in{\mathbb Z}$ and $\alpha\uparrow\lambda\in{\mathcal J}$. Then  there is a short exact sequence
$$
0\to{\overline{\Delta}} _{\mathfrak p}(\alpha\uparrow\lambda)\to{\overline{P}}_{\mathfrak p}^{{\mathcal J}}(\lambda)\to{\overline{\Delta}} _{\mathfrak p}(\lambda)\to 0.
$$
\end{enumerate}
\end{Fiebigproposition}

If ${\overline{P}}_{\mathfrak p}^{{\mathcal J}}(\lambda)\cong{\overline{\Delta}} _{\mathfrak p}(\lambda)$, then clearly $\operatorname{End}_{{\overline{\mathcal O}} _{\mathfrak p}}({\overline{P}}_{\mathfrak p}^{{\mathcal J}}(\lambda))={{\widetilde S}}_{\mathfrak p}$. We now describe the endomorphism ring in case ${\overline{P}}_{\mathfrak p}^{{\mathcal J}}(\lambda)$ has a two step Verma flag as in part (3) of the proposition above. Let us set $Q={{\widetilde S}}_{(0)}$.
By Theorem \ref{thm-propproj}, ${\overline{P}}_{\mathfrak p}^{{\mathcal J}}(\lambda)\otimes_{{\widetilde S}} Q$ is projective in ${\overline{\mathcal O}} _Q^{{\mathcal J}}$, and by the proposition above, it splits into a direct sum of copies of restricted Verma modules. Then ${\overline{P}}_{\mathfrak p}^{{\mathcal J}}(\lambda)\otimes_{{\widetilde S}} Q\cong{\overline{\Delta}} _Q(\lambda)\oplus{\overline{\Delta}} _Q(\alpha\uparrow\lambda)$ and hence $\operatorname{End}_{{\overline{\mathcal O}} _Q}({\overline{P}}_{\mathfrak p}^{{\mathcal J}}(\lambda)\otimes_{{\widetilde S}} Q)\cong Qe_\lambda\oplus Qe_{\alpha\uparrow\lambda}$, where $e_\lambda$ and $e_{\alpha\uparrow\lambda}$ are the two idempotents associated with the preceding direct sum decomposition. We have a natural inclusion 
$$
\operatorname{End}_{{\overline{\mathcal O}} _{\mathfrak p}}({\overline{P}}_{\mathfrak p}^{{\mathcal J}}(\lambda))\subset \operatorname{End}_{{\overline{\mathcal O}} _Q}({\overline{P}}_{\mathfrak p}^{{\mathcal J}}(\lambda)\otimes_{{\widetilde S}} Q)=Q\oplus Q.
$$ 
Here is an important structural result. Its proves uses the Jantzen filtration and the Jantzen sum formula. 
 \begin{Fiebigproposition}[{cf.\  \cite[Proposition 2]{K}}]\label{prop-subgenend} The image of the above inclusion is $\{(z_\lambda,z_{\alpha\uparrow\lambda})\in {\widetilde S}\oplus {\widetilde S}\mid z_\lambda\equiv z_{\alpha\uparrow\lambda}\mod\alpha^\vee\}.$ 
 \end{Fiebigproposition}

\section{Subquotient categories} \label{sec-subquot}
In this section we define subquotient categories of ${\overline{\mathcal O}} _{A,\Lambda}$. We start with recalling the general framework as it appears in \cite{Gab}.
\subsection{Generalities}
Let ${\mathcal{A}}$ be an abelian category and ${\mathcal{N}}\subset{\mathcal{A}}$ a Serre subcategory, i.e.\  a subcategory that has the property that for all short exact sequences $0\to M_1\to M_2\to M_3\to 0$ we have $M_2\in{\mathcal{N}}$ if and only if $M_1,M_3\in{\mathcal{N}}$. One then defines the quotient category ${\mathcal{A}}/{\mathcal{N}}$ as follows. The objects of ${\mathcal{A}}/{\mathcal{N}}$ are the objects of ${\mathcal{A}}$, and $\operatorname{Hom}_{{\mathcal{A}}/{\mathcal{N}}}(M,N)$ is the direct limit of $\operatorname{Hom}_{{\mathcal{A}}}(M^\prime, N/N^\prime)$, where $M^\prime$ and $N^\prime$ are subobjects of $M$ and $N$, resp., such that $M/M^\prime$ and $N^\prime$ are contained in ${\mathcal{N}}$. We denote by $T\colon{\mathcal{A}}\to{\mathcal{A}}/{\mathcal{N}}$ the obvious functor.

In \cite{Gab} the following is proven:

\begin{Fiebigtheorem} \label{thm-propquot}\begin{enumerate}
\item  The quotient category ${\mathcal{A}}/{\mathcal{N}}$ is  abelian and the quotient functor $T\colon{\mathcal{A}}\to{\mathcal{A}}/{\mathcal{N}}$ is exact.
\item  If $(\ast)\quad 0\to M_1\to M_2\to M_3\to 0$ is a short exact sequence in ${\mathcal{A}}/{\mathcal{N}}$, then there exists a short exact sequence $0\to M_1^\prime\to M_2^\prime\to M_3^\prime\to 0$ in ${\mathcal{A}}$ such that $0\to T(M_1^\prime)\to T(M_2^\prime)\to T(M_3^\prime)\to 0$ is isomorphic to $(\ast)$.
\item We have $TM\cong 0$ if and only if $M$ is an object in ${\mathcal{N}}$.
\end{enumerate}
\end{Fiebigtheorem}

\subsection{Subquotient categories of ${\overline{\mathcal O}} _{A,\Lambda}$}

Let $A$ be a deformation algebra and $\Lambda$ a $\sim_A$-equivalence class in ${\widehat{\mathfrak h}^\star}_{\operatorname{crit}}$. Let ${\mathcal K}\subset\Lambda$ be a locally closed and locally bounded subset. Recall that ${\mathcal K}={\mathcal K}_+\setminus{\mathcal K}_-$, where ${\mathcal K}_+$ is the smallest open subset that contains ${\mathcal K}$, and ${\mathcal K}_-\subset{\mathcal K}_+$ is the largest open subset which intersects ${\mathcal K}$ trivially.  Then ${\overline{\mathcal O}} _{A,\Lambda}^{{\mathcal K}_-}$ is a Serre subcategory of ${\overline{\mathcal O}} _{A,\Lambda}^{{\mathcal K}_+}$ and we define 
$$
{\overline{\mathcal O}} ^{[{\mathcal K}]}_{A,\Lambda}:={\overline{\mathcal O}} _{A,\Lambda}^{{\mathcal K}_+}/{\overline{\mathcal O}} _{A,\Lambda}^{{\mathcal K}_-}.
$$ 
As before we denote by $T$ the natural quotient functor ${\overline{\mathcal O}} _{A,\Lambda}^{{\mathcal K}_+}\to{\overline{\mathcal O}} _{A,\Lambda}^{[{\mathcal K}]}$.

\begin{Fiebiglemma} 
Let $M$ and $N$ be objects in ${\overline{\mathcal O}} _{A,\Lambda}^{{\mathcal K}_+}$. 
\begin{enumerate}
\item There is a minimal submodule $M_+$ of $M$ such that $M/M_+$ is contained in ${\overline{\mathcal O}} _{A,\Lambda}^{{\mathcal K}_-}$. 
\item There is a maximal submodule $N_-$ of $N$ that is contained in ${\overline{\mathcal O}} _{A,\Lambda}^{{\mathcal K}_-}$.
\item The functorial homomorphism $\operatorname{Hom}_{{\overline{\mathcal O}} _{A,\Lambda}^{{\mathcal K}_+}}(M,N)\to\operatorname{Hom}_{{\overline{\mathcal O}} _{A,\Lambda}^{[{\mathcal K}]}}(TM,TN)$ factors over the natural homomorphism $\operatorname{Hom}_{{\overline{\mathcal O}} _{A,\Lambda}^{{\mathcal K}_+}}(M,N)\to \operatorname{Hom}_{{\overline{\mathcal O}} _{A,\Lambda}^{{\mathcal K}_+}}(M_+,N/N_-)$ and induces an isomorphism
$$
\operatorname{Hom}_{{\overline{\mathcal O}} _{A,\Lambda}^{{\mathcal K}_+}}(M_+,N/N_-)\stackrel{\sim}\to \operatorname{Hom}_{{\overline{\mathcal O}} _{A,\Lambda}^{[{\mathcal K}]}}(TM,TN)
$$
of vector spaces.
\end{enumerate}
\end{Fiebiglemma} 
\begin{proof} The object $M_+$ is the submodule of $M$ generated by all weight spaces with weights in ${\mathcal K}$, and the  object $M_-$ is the sum of all  submodules  of $M$ with weights contained in ${\mathcal K}_-$. The statement in (3) the follows immediately from (1) and (2) and the definition.
\end{proof}

\subsection{Projectives in ${\overline{\mathcal O}} _{A,\Lambda}^{[{\mathcal K}]}$}
The following results concern easy to prove properties of the projectives in our subquotient categories.

\begin{Fiebiglemma}  \label{lemma-projinquot} Let $\mu\in{\mathcal K}$. Then the functorial homomorphism
$$
\operatorname{Hom}_{{\overline{\mathcal O}} _{A,\Lambda}^{{\mathcal K}_+}}({\overline{P}}^{{\mathcal K}_+}_A(\mu),N)\to \operatorname{Hom}_{{\overline{\mathcal O}} _{A,\Lambda}^{[{\mathcal K}]}}(T{\overline{P}}^{{\mathcal K}_+}_A(\mu),TN)
$$ 
is an isomorphism for all objects $N$ in ${\overline{\mathcal O}} _{A,\Lambda}^{{\mathcal K}_+}$.
\end{Fiebiglemma} 

\begin{proof} Note that ${\overline{P}}^{{\mathcal K}_+}_A(\mu)_+={\overline{P}}^{{{\mathcal K}_+}}_A(\mu)$ as ${\overline{P}}^{{\mathcal K}_+}_A(\mu)$ is generated by its $\mu$-weight space. Hence $\operatorname{Hom}_{{\overline{\mathcal O}} _{A,\Lambda}^{[{\mathcal K}]}}(T{\overline{P}}^{{\mathcal K}_+}_A(\mu),TN)=\operatorname{Hom}_{{\overline{\mathcal O}} _{A,\Lambda}^{{\mathcal K}_+}}({\overline{P}}^{{\mathcal K}_+}_A(\mu),N/N_-)$. As ${\overline{P}}^{{{\mathcal K}_+}}_A(\mu)$ is projective in ${\overline{\mathcal O}} ^{{{\mathcal K}_+}}_{A,\Lambda}$ and as $\operatorname{Hom}_{{\overline{\mathcal O}} ^{{{\mathcal K}_+}}_{A,\Lambda}}({\overline{P}}^{{\mathcal K}_+}_A(\mu),N_-)=0$ (as $(N_-)_\mu=0$), the claim follows. 
\end{proof}

\begin{Fiebigproposition}\label{prop-propquot} 
Let $\mu\in{\mathcal K}$. Then $T{\overline{P}}^{{\mathcal K}_+}_A(\mu)$ is projective in ${\overline{\mathcal O}} _{A,\Lambda}^{[{\mathcal K}]}$. 
\end{Fiebigproposition}
\begin{proof} We have to show that $\operatorname{Hom}_{{\overline{\mathcal O}} _{A,\Lambda}^{[{\mathcal K}]}}(T {\overline{P}}^{{\mathcal K}_+}_A(\mu),\cdot)$ maps a short exact sequence in ${\overline{\mathcal O}} _{A,\Lambda}^{[{\mathcal K}]}$ to a short exact sequence of abelian groups. By Theorem \ref{thm-propquot} is suffices to check this property on short exact sequences in the image of $T$. For those we can replace $\operatorname{Hom}_{{\overline{\mathcal O}} _{A,\Lambda}^{[{\mathcal K}]}}(T {\overline{P}}^{{\mathcal K}_+}_A(\mu),T \cdot)$  by $\operatorname{Hom}_{{\overline{\mathcal O}} _{A,\Lambda}^{{\mathcal K}_+}}({\overline{P}}^{{\mathcal K}_+}_A(\mu),\cdot)$ by Lemma \ref{lemma-projinquot}. So the projectivity of ${\overline{P}}^{{\mathcal K}_+}_A(\mu)$ in ${\overline{\mathcal O}} _{A,\Lambda}^{{\mathcal K}_+}$ yields the statement. 
\end{proof}

\subsection{Objects admitting a Verma flag}
We denote by ${\widetilde{\Delta}} _A(\mu):= T {\overline{\Delta}} _A(\mu)$ for $\mu\in{\mathcal K}$ the Verma module in the subquotient category ${\overline{\mathcal O}} _{A,\Lambda}^{[{\mathcal K}]}$.

\begin{Fiebigdefinition} We say that an object $M$ of ${\overline{\mathcal O}} _{A,\Lambda}^{[{\mathcal K}]}$ {\em admits a Verma flag} if there is a finite filtration $0=M_0\subset M_1\subset\dots\subset M_n=M$ such that $M_i/M_{i-1}$ is isomorphic to ${\widetilde{\Delta}} _A(\mu_i)$ for some $\mu_i\in{\mathcal K}$.
\end{Fiebigdefinition}
We denote by ${\overline{\mathcal O}} _{A,\Lambda}^{[{\mathcal K}],V}$ the full subcategory of ${\overline{\mathcal O}} _{A,\Lambda}^{[{\mathcal K}]}$ that contains all objects that admit a Verma flag.
If $M\in{\overline{\mathcal O}} _{A,\Lambda}^{[{\mathcal K}]}$ admits a Verma flag, then so does $TM$, as the functor $T$ is exact. In particular, the objects $T{\overline{P}}^{{\mathcal K}_+}_A(\mu)$ admit a Verma flag. 
\section{Moment graphs}\label{sec-momgra}
Now we introduce the ``combinatorial part'' of the picture. It has a priori nothing to do with Lie algebras or their representations. More information on the following constructions can be found in \cite{FieAdv}.
\subsection{Moment graphs associated with equivalence classes}
Let $\Lambda$ be a $\sim_{{\widetilde S}}$-equivalence class in $\Lambda$. Recall that we associated with $\Lambda$ the set of $\Lambda$-integral roots $\widehat R_\Lambda$ and the $\Lambda$-integral Weyl group $\widehat{{\mathcal W}}_\Lambda$. We now associate a moment graph with $\Lambda$.  It is a moment graph over the lattice $X^\vee\subset{\mathfrak h}$ of (finite) coweights. 

\begin{Fiebigdefinition} The moment graph ${\mathcal{G}}_{\Lambda}$ associated with $\Lambda$ is given as follows:
\begin{itemize}
\item The set of vertices is $\Lambda$. It is partially order by the order $\le$ on $\Lambda$. 
\item The vertices $\lambda,\mu\in\Lambda$ are connected by an edge if there is a  $\Lambda$-integral root $\beta\in \widehat R_\Lambda$ with  $\mu=s_\beta.\lambda$. The edge connecting $\lambda$ and $\mu$ is then labelled by $\pm{\overline{\beta^\vee}}\in X^\vee/\{\pm1\}$ (recall that we denote by $H\mapsto \overline H$ the map $\widehat{\mathfrak{h}}\to\mathfrak{h}$).
\end{itemize}
\end{Fiebigdefinition}
For any subset ${{\mathcal K}}$ of $\Lambda$ we denote by ${\mathcal{G}}_{{\mathcal K}}$ the full  sub-moment graph of ${\mathcal{G}}_\Lambda$ that contains all vertices in ${{\mathcal K}}$. 
\begin{Fiebiglemma}  Suppose that ${{\mathcal K}}$ has the property that  $\mu\in {{\mathcal K}}$ implies $\mu+n\delta\not\in {{\mathcal K}}$ for all $n\ne 0$. Then $({\mathcal{G}}_{{\mathcal K}},{\mathbb C})$ satisfies the GKM-condition.
\end{Fiebiglemma} 
Recall that the GKM-condition says that for any two disctinct edges $E$ and $\tilde E$ of ${\mathcal{G}}_{{\mathcal K}}$ with labels $\beta^\vee$ and $\tilde \beta^\vee$ that are adjacent to a common vertex we have $\beta^\vee\not\in\mathbb C\tilde\beta^\vee$.
\begin{proof} As distinct positive coroots are ${\mathbb C}$-linearly independent, it suffices to show that there is no $\lambda\in {{\mathcal K}}$ such that $s_{\alpha+n\delta}.\lambda$ and $s_{\alpha+m\delta}.\lambda$ are in ${{\mathcal K}}$ for distinct $m,n\in{\mathbb Z}$. As $s_{\alpha+n\delta}.\lambda-s_{\alpha+m\delta}.\lambda$ is a multiple of $\delta$ (as $\lambda\in {\widehat{\mathfrak h}}^\star_{\operatorname{crit}}$), this follows immediately from our assumption on ${{\mathcal K}}$.
\end{proof}

\subsection{The structure algebra of ${\mathcal{G}}_{{\mathcal K}}$}
We will study modules over the following commutative $S$-algebra.
\begin{Fiebigdefinition} The algebra 
$$
{\mathcal{Z}}({{\mathcal K}}):=\left\{(z_\mu)\in\prod_{\mu\in {{\mathcal K}}} S\left|\begin{matrix} z_\mu\equiv z_\lambda\mod\alpha^\vee\\
\text{ for all edges $\lambda\stackrel{\alpha^\vee}{\text{---\!\!\!---\!\!\!---}}\mu$ of ${\mathcal{G}}_{{\mathcal K}}$}
\end{matrix}
\right\}\right.
$$
is called the {\em structure algebra of ${\mathcal{G}}_{{\mathcal K}}$}. 
\end{Fiebigdefinition}
For a deformation algebra $A$ we define ${\mathcal{Z}}_A({{\mathcal K}}):={\mathcal{Z}}({{\mathcal K}})\otimes_S A$. We denote by ${\mathcal{Z}}_A({{\mathcal K}}){\operatorname{-mod}}^f$ the category of ${\mathcal{Z}}_A({{\mathcal K}})$-modules that are torsion free and finitely generated as $A$-modules. For any homomorphism $A\to A^\prime$ we obtain a functor
\begin{align*}
{\mathcal{Z}}_A({{\mathcal K}}){\operatorname{-mod}}^f&\to{\mathcal{Z}}_{A^\prime}({{\mathcal K}}){\operatorname{-mod}}^f\\
M&\mapsto M\otimes_A A^\prime.
\end{align*}

Let $\mu$ be an element of ${{\mathcal K}}$. The following is the {\em standard object} with parameter $\mu$ in ${\mathcal{Z}}_A{\operatorname{-mod}}^f$. 
\begin{Fiebigdefinition} We denote by ${\mathcal V}_A(\mu)$ the ${\mathcal{Z}}_A({{\mathcal K}})$-module that is free of rank $1$ as an $A$-module and on which $(z_\nu)\in{\mathcal{Z}}_A({{\mathcal K}})$ acts as multiplication with $z_\mu$.
\end{Fiebigdefinition}

\subsection{Generic decomposition} Suppose now that ${{\mathcal K}}\subset\Lambda$ is a finite set. 
It is easy to prove that the inclusion ${\mathcal{Z}}_{{\widetilde S}}({{\mathcal K}})\subset\bigoplus_{\mu\in{\mathcal K}} {\widetilde S}$ becomes a bijection after tensoring with $Q={{\widetilde S}}_{(0)}$, i.e.\  we canonically have
$$
{\mathcal{Z}}_Q({{\mathcal K}})=\bigoplus_{\mu\in{\mathcal K}}Q.
$$
Hence for any object $N$ in ${\mathcal{Z}}_{{\widetilde S}}({{\mathcal K}}){\operatorname{-mod}}^f$ we have a canonical decomposition $N_Q=\bigoplus N_Q^\mu$. In particular,
 if $M$ is a ${\mathcal{Z}}_{{\widetilde S}}({{\mathcal K}})$-module that is torsion free over ${\widetilde S}$, then $M\subset M_Q=\bigoplus_{\mu\in{\mathcal K}} M_Q^{\mu}$.

\begin{Fiebigdefinition} Let ${\mathcal I}$ be a closed subset in ${{\mathcal K}}$ with open complement ${\mathcal J}$. For a ${\mathcal{Z}}_{{\widetilde S}}({{\mathcal K}})$-module $M$ that is torsion free over ${\widetilde S}$ we define
$$
M_{\mathcal I}:=M\cap \bigoplus_{\mu\in{\mathcal I}} M_Q^\mu
$$
and
$$
M^{\mathcal J}:=M/M_{\mathcal I}.
$$
\end{Fiebigdefinition}

\subsection{The category ${\mathcal C}_{{\widetilde S}}({{\mathcal K}})$}\label{sec-defC}

Let $M$ be an object in ${\mathcal{Z}}_{{\widetilde S}}({{\mathcal K}}){\operatorname{-mod}}^f$. 
\begin{Fiebigdefinition} We say that $M$ admits a Verma flag if $M^{\mathcal J}$ is free as an ${\widetilde S}$-module for each open subset ${\mathcal J}$ of ${{\mathcal K}}$.
\end{Fiebigdefinition}

We denote by ${\mathcal C}_{{\widetilde S}}({{\mathcal K}})$ the full subcategory of ${\mathcal{Z}}_{{\widetilde S}}({{\mathcal K}}){\operatorname{-mod}}^f$ that contains all objects that admit a Verma flag. This category carries a natural  but non-standard exact structure that we will utilize in the following.

\begin{Fiebigdefinition} Let $0\to M_1\to M_2\to M_3\to 0$ be a sequence in ${\mathcal C}_{{\widetilde S}}({{\mathcal K}})$. We say that it is {\em short exact}, if for any open subset ${\mathcal J}$ the induced sequence
$$
0\to (M_1)^{\mathcal J}\to (M_2)^{\mathcal J}\to (M_3)^{\mathcal J}\to 0
$$
is a short exact sequence of abelian groups. 
\end{Fiebigdefinition}
This defines indeed an exact structure (cf.\   \cite{FieAdv}), so the notion of {\em projective object} now makes sense in ${\mathcal C}_{{\widetilde S}}({\mathcal K})$.

\subsection{Projective objects in ${\mathcal C}_{{\widetilde S}}({{\mathcal K}})$}

Let $M$ be an object in  ${\mathcal C}_{{\widetilde S}}({{\mathcal K}})$. Then it is free as an ${\widetilde S}$-module, hence we have an inclusion $M\subset M_Q=\bigoplus_{\mu\in {\mathcal K}} M_Q^\mu$. We denote by $M^\mu$ the image of the homomorphism $M\to M_Q\stackrel{\pi_\mu}\to M_Q^\mu$. Let $E\colon\chi\stackrel{\alpha^\vee}{\text{---\!\!\!---\!\!\!---}} \chi^\prime$ be an edge. We let $M^{\chi,\chi^\prime}$ be the image of the homomorphism $M\to M_Q\stackrel{(\pi_\chi,\pi_{\chi^\prime})}\to M_Q^\chi\oplus M_Q^{\chi^\prime}$. We define the {\em local structure algebra at $E$} as
$$
{\mathcal{Z}}_{{\widetilde S}}(E)=\{(z_\chi,z_{\chi^\prime})\in {\widetilde S} \oplus {\widetilde S}\mid z_\chi\equiv z_{\chi^\prime}\mod\alpha^\vee\}.
$$
Finally, let $M^E$ be the ${\mathcal{Z}}_{{\widetilde S}}(E)$-submodule in $M^{\chi}\oplus M^{\chi^\prime}$ that is generated by $M^{\chi,\chi^\prime}$. 
\begin{Fiebigproposition}[\cite{FieAdv}]\label{prop-projinC} Suppose that $M$ satsifies the following properties:
\begin{enumerate}
\item $M^\chi$ is a free ${\widetilde S}$-module for all $\chi\in \mathcal K$,
\item For an edge $E\colon \chi\stackrel{\alpha^\vee}{\text{---\!\!\!---\!\!\!---}}\chi^\prime$ with $\chi<\chi^\prime$ we have $(M^E)_\chi:=M^E\cap M^\chi=\alpha^\vee M^{\chi}$. 
\end{enumerate}
Then $M$ is projective in ${\mathcal C}_{{\widetilde S}}({{\mathcal K}})$.
\end{Fiebigproposition}

\section{Soergel theory}\label{sec-ST}
We now bring the representation theory and the moment graph theory together using Soergel's {\em Strukturfunktor} (cf.\  \cite{SJAMS}). We prove versions of the Struktursatz and the Endomorphismensatz of loc.\ cit. Our framework is the following. As before we fix a $\sim_{{\widetilde S}}$-equivalence class $\Lambda$. 
Let ${\mathcal J}$ be an open and locally bounded subset of $\Lambda$ and $\lambda\in{\mathcal J}$. We set ${\mathcal K}:={\mathcal J}\cap\{\ge\lambda\}$. Then ${\mathcal K}$ is locally closed and locally bounded. From now on we assume that the pair $({\mathcal J},\lambda)$ satisfies the following two assumptions:
\begin{itemize}
\item The  object ${\overline{P}}_{{\widetilde S}}^{\mathcal J}(\lambda)$ is multiplicity free, i.e.\  $({\overline{P}}_{{\widetilde S}}^{\mathcal J}(\lambda):{\overline{\Delta}} _{{\widetilde S}}(\mu))\le 1$ for all $\mu\in{\mathcal J}$.
\item We have $ \operatorname{supp}_{{\overline{\Delta}} }\, {\overline{P}}_{{\widetilde S}}^{\mathcal J}(\lambda)={\mathcal K}$.
\item  If  $\mu\in {\mathcal K}$, then   $\mu+n\delta\not\in{\mathcal K}$ for all $n\ne 0$.
\end{itemize}
From now on we fix the data above.

\subsection{The Endomorphismensatz}
Recall that we assume that ${\overline{P}}_{{\widetilde S}}^{\mathcal J}(\lambda)$ is multiplicity free. This implies that ${\overline{P}}_{{\widetilde S}}^{\mathcal J}(\lambda)\otimes_{{\widetilde S}} Q$ splits into a direct sum of non-isomorphic Verma modules. As there are no non-trivial homomorphisms ${\overline{\Delta}} _Q(\mu)\to{\overline{\Delta}} _Q(\mu^\prime)$ for $\mu\ne\mu^\prime$, we get a canonical decomposition
$$
\operatorname{End}_{{\overline{\mathcal O}} _Q}({\overline{P}}_{{\widetilde S}}^{\mathcal J}(\lambda)\otimes_{{\widetilde S}} Q)=\bigoplus_{\mu\in {{\mathcal K}}} Q.
$$
As ${\overline{P}}_{{\widetilde S}}^{\mathcal J}(\lambda)$ admits a Verma flag, it is free as an ${\widetilde S}$-module. Hence $\operatorname{End}_{{\overline{\mathcal O}} _{{\widetilde S}}}({\overline{P}}_{{\widetilde S}}^{{\mathcal J}}(\lambda))$ is torsion free as an ${\widetilde S}$-module. Then the canonical homomorphism $\operatorname{End}_{{\overline{\mathcal O}} _{{\widetilde S}}}({\overline{P}}_{{\widetilde S}}^{{\mathcal J}}(\lambda))\to \operatorname{End}_{{\overline{\mathcal O}} _{{\widetilde S}}}({\overline{P}}_{{\widetilde S}}^{{\mathcal J}}(\lambda))\otimes_{{\widetilde S}} Q$ is injective. But by Theorem \ref{thm-propproj} the latter $Q$-module is canonically isomorphic to $\operatorname{End}_{{\overline{\mathcal O}} _Q}({\overline{P}}_{{\widetilde S}}^{\mathcal J}(\lambda)\otimes_{{\widetilde S}} Q)$, so we obtain a canonical inclusion
$$
\operatorname{End}_{{\overline{\mathcal O}} _{{\widetilde S}}}({\overline{P}}_{{\widetilde S}}^{{\mathcal J}}(\lambda))\to\bigoplus_{\mu\in {\mathcal K}} Q.
$$


\begin{Fiebigtheorem}  The image of the inclusion above is ${\mathcal{Z}}_{{\widetilde S}}({\mathcal K})$, hence we obtain a canonical isomorphism 
$$
\operatorname{End}_{{\overline{\mathcal O}} _{{\widetilde S}}}({\overline{P}}_{{\widetilde S}}^{\mathcal J}(\lambda))\cong {\mathcal{Z}}_{{\widetilde S}}({{\mathcal K}}).
$$
of ${\widetilde S}$-algebras. 
\end{Fiebigtheorem} 

\begin{proof}  By Lemma \ref{lemma-bcone} we have 
$$
\operatorname{End}_{{\overline{\mathcal O}} _{{\widetilde S}}}({\overline{P}}^{\mathcal J}_{{\widetilde S}}(\lambda))=\bigcap_{{\mathfrak p}\in{\mathfrak P}}\operatorname{End}_{{\overline{\mathcal O}} _{\mathfrak p}}({\overline{P}}^{\mathcal J}_{{\widetilde S}}(\lambda)\otimes_{{\widetilde S}} {{\widetilde S}}_{\mathfrak p})\subset \operatorname{End}_{{\overline{\mathcal O}} _Q}({\overline{P}}^{\mathcal J}_{{\widetilde S}}(\lambda)\otimes_{{\widetilde S}} Q).
$$
Now ${\overline{P}}^{\mathcal J}_{{\widetilde S}}(\lambda)\otimes_{{\widetilde S}} {{\widetilde S}}_{\mathfrak p}$ is projective in ${\overline{\mathcal O}} _{\mathfrak p}$. If ${\mathfrak p}$ is a prime ideal of height one,  Proposition \ref{prop-strucproj} implies that ${\overline{P}}^{\mathcal J}_{{\widetilde S}}(\lambda)\otimes_{{\widetilde S}} {{\widetilde S}}_{\mathfrak p}$ splits into indecomposable direct summands that are either isomorphic to Verma modules, or admit a two step Verma filtration. We deduce that $\operatorname{End}({\overline{P}}^{\mathcal J}_{{\widetilde S}}(\lambda)\otimes_{{\widetilde S}} {{\widetilde S}}_{\mathfrak p})\subset\bigoplus {{\widetilde S}}_{\mathfrak p}$, and for $(z_\mu)\in\operatorname{End}({\overline{P}}^{\mathcal J}_{{\widetilde S}}(\lambda)\otimes_{{\widetilde S}} {{\widetilde S}}_{\mathfrak p})$  we have $z_\mu\cong z_{\alpha\uparrow\mu}\mod\alpha^\vee$ in case the direct summand ${\overline{P}}^{\mathcal J}_{\mathfrak p}(\mu)$ with a two step Verma filtration occurs (by Proposition \ref{prop-subgenend}). Now taking the intersection over all ${\mathfrak p}$ of height one yields the claim. 
\end{proof}

\subsection{The Strukturfunktor}

Let $\Lambda$, ${\mathcal J}$ and $\lambda$  be as above.  Using the Endomorphismensatz we  can identify the category ${\operatorname{mod-}}\operatorname{End}_{{\overline{\mathcal O}} _{{\widetilde S},\Lambda}}({\overline{P}}^{\mathcal J}_{{\widetilde S}}(\lambda))$ with ${\mathcal{Z}}_{\widetilde S}({{\mathcal K}}){\operatorname{-mod}}$.

\begin{Fiebigdefinition} The {\em Strukturfunktor} associated with the above data is defined as 
$$
{\mathbb V}:=\operatorname{Hom}_{{\overline{\mathcal O}} _{{\widetilde S},\Lambda}}({\overline{P}}^{\mathcal J}_{{\widetilde S}}(\lambda),\cdot)\colon{\overline{\mathcal O}} ^{\mathcal J}_{{\widetilde S},\Lambda}\to{\mathcal{Z}}_{\widetilde S}({{\mathcal K}}){\operatorname{-mod}}.
$$
\end{Fiebigdefinition}
We collect some properties of ${\mathbb V}$.
\begin{Fiebigproposition}\label{prop-propV}  
\begin{enumerate}
\item If $M\in{\overline{\mathcal O}}^{\mathcal J}_{{\widetilde S},\Lambda}$ admits a Verma flag, then so does  ${\mathbb V} M$. Hence we obtain a functor
$$
{\mathbb V}\colon {\overline{\mathcal O}} ^{\mathcal J,V}_{{\widetilde S},\Lambda}\to{\mathcal C}_{\widetilde S}({{\mathcal K}}).
$$
\item The functor ${\mathbb V}$ is exact.
\item For $\mu\in{{\mathcal K}}$ we have ${\mathbb V}({\overline{\Delta}} _{{\widetilde S}}(\mu))\cong{\mathcal V}_{{\widetilde S}}(\mu)$.
\end{enumerate}
\end{Fiebigproposition}
\begin{proof} 
Let us prove statement (3). We have ${\overline{P}}_{{\widetilde S}}^{\mathcal J}(\lambda)\otimes_{{\widetilde S}} Q\cong\bigoplus_{\mu\in{{\mathcal K}}}{\overline{\Delta}} _Q(\mu)$, and an element  $z=(z_\mu)\in{\mathcal{Z}}$ acts diagonally by multiplication with the coordinates in $S$. Using the base change result in Theorem \ref{thm-propproj} we deduce  that $z$ acts on ${\mathbb V}({\overline{\Delta}} _{{\widetilde S}}(\mu))$ as multiplication with $z_\mu$. So we only have to show that ${\mathbb V}(\Delta_{{\widetilde S}}(\mu))$ is free of rank one as a ${\widetilde S}$-module. As ${\widehat{P}} _{\mathbb C}^{\mathcal J}(\lambda)$ is a projective cover of $L_{\mathbb C}(\lambda)$, we have 
\begin{align*}
\dim_{\mathbb C} \mathbb V(\overline\Delta_{\widetilde S}(\mu))\otimes_{\widetilde S}\mathbb C&=\dim_{\mathbb C} \operatorname{Hom}_{{\overline{\mathcal O}} _{{\widetilde S},\Lambda}}({\overline{P}}^{\mathcal J}_{{\widetilde S}}(\lambda),\overline\Delta_{\widetilde S}(\mu))\otimes_{\widetilde S}\mathbb C\\
&=\dim_{\mathbb C} \operatorname{Hom}_{{\overline{\mathcal O}} _{\mathbb C}^{\mathcal J}}(\overline{P}_{\widetilde S}^{\mathcal J}(\lambda)\otimes_{\widetilde S}\mathbb C,{\overline{\Delta}} _{\mathbb C}(\mu))\\
&=[{\overline{\Delta}} _{\mathbb C}(\mu), L_{\mathbb C}(\lambda)].
\end{align*}

 By Theorem \ref{thm-propproj} we have $[{\overline{\Delta}} _{\mathbb C}(\mu), L_{\mathbb C}(\lambda)]=({\overline{P}}_{\mathbb C}^{\mathcal J}(\lambda):{\overline{\Delta}} _{\mathbb C}(\mu))$. By assumption, the latter number is $1$. So by  Nakayama's lemma we deduce that ${\mathbb V}{\overline{\Delta}} _{{\widetilde S}}(\mu))$ is cyclic as an ${\widetilde S}$-module. As ${\overline{\Delta}} _{{\widetilde S}}(\mu)$ is torsion free (as an ${\widetilde S}$-module), so is ${\mathbb V}({\overline{\Delta}} _{{\widetilde S}}(\mu))$. Hence ${\mathbb V}({\overline{\Delta}} _{{\widetilde S}}(\mu))$ is free of rank $1$.

We prove statement (1). If $M$ is an object in ${\overline{\mathcal O}} _{{\widetilde S},\Lambda}^{V}$, then its quotient $M^{\mathcal J}$ and the submodule $M_{\mathcal I}$ admit Verma flags as well. As ${\overline{P}}_{{\widetilde S}}^{\mathcal J}(\lambda)$ is projective in ${\overline{\mathcal O}} _{{\widetilde S},\Lambda}^{{\mathcal J}}$, the functor ${\mathbb V}$ is exact when we consider the natural abelian structure on ${\mathcal{Z}}_{{\widetilde S}}({{\mathcal K}}){\operatorname{-mod}}$. Using (3) we deduce that ${\mathbb V}(M_{\mathcal I})={\mathbb V}(M)_{{\mathcal I}}$ and ${\mathbb V}(M^{\mathcal J})={\mathbb V}(M)^{{\mathcal J}}$. Moreover, both objects are free as ${\widetilde S}$-modules. Hence ${\mathbb V}(M)$ admits a Verma flag. 

We are left with proving statement (2). 
Suppose that $0\to M_1\to M_2\to M_3\to 0$ is a short exact sequence in ${\overline{\mathcal O}} _{{\widetilde S},\Lambda}^{{\mathcal J},V}$. Then  for each open subset ${\mathcal J}^\prime$ the sequence
$$
0\to M_1^{{\mathcal J}^\prime}\to M_2^{{\mathcal J}^\prime}\to M_3^{{\mathcal J}^\prime}\to0
$$
 is exact as well. By the projectivity of ${\overline{P}}_{{\widetilde S}}^{\mathcal J}(\lambda)$, the sequence
 $$
 0\to{\mathbb V}(M_1^{{\mathcal J}^\prime})\to {\mathbb V}(M_2^{{\mathcal J}^\prime})\to {\mathbb V}(M_3^{{\mathcal J}^\prime})\to 0
 $$
 is exact as a sequence of ${\mathcal{Z}}_{{\widetilde S}}(\mathcal K)$-modules (with the standard exact structure). By what we established above, this sequence  identifies with the sequence 
 $$
 0\to({\mathbb V} M_1)^{{\mathcal J}^\prime}\to ({\mathbb V} M_2){{\mathcal J}^\prime}\to ({\mathbb V} M_3)^{{\mathcal J}^\prime}\to 0.
 $$
 Hence ${\mathbb V}$ is an exact functor. 
\end{proof}

\subsection{The Struktursatz}
We now show that the Strukturfunktor is fully faithful on the relevant objects. 
\begin{Fiebigtheorem}  \label{thm-ff} Let $\mu\in{{\mathcal K}}$. Then the functorial homomorphism
$$
\operatorname{Hom}_{{\overline{\mathcal O}} _{{\widetilde S},\Lambda}}({\overline{P}}^{\mathcal J}_{{\widetilde S}}(\mu),M)\to\operatorname{Hom}_{{\mathcal C}_{{\widetilde S}}({{\mathcal K}})}({\mathbb V} {\overline{P}}^{\mathcal J}_{{\widetilde S}}(\mu),{\mathbb V} M)
$$
is an isomorphism for any object $M$ of  ${\overline{\mathcal O}}^{{\mathcal J},V}_{{\widetilde S},\Lambda}$.  
\end{Fiebigtheorem}  
\begin{proof}  For notational simplicity we now write $P$ instead of  ${\overline{P}}^{\mathcal J}_{{\widetilde S}}(\mu)$. By Lemma \ref{lemma-bcone} we have 
$$
\operatorname{Hom}_{{\overline{\mathcal O}} _{{\widetilde S}}}(P,M)=\bigcap_{{\mathfrak p}\in{\mathfrak P}}  \operatorname{Hom}_{{\overline{\mathcal O}} _{\mathfrak p}}(P_{\mathfrak p},M_{\mathfrak p}).
$$
Analogously (and with the same proof), we have 
$$
\operatorname{Hom}_{{\mathcal{Z}}_{{\widetilde S}}({\mathcal K})}({\mathbb V} P,{\mathbb V} M)=\bigcap_{{\mathfrak p}\in{\mathfrak P}}  \operatorname{Hom}_{{\mathcal{Z}}_{\mathfrak p}({\mathcal K})}(({\mathbb V} P)_{\mathfrak p},({\mathbb V} M)_{\mathfrak p}).
$$
By the base change result in Proposition \ref{prop-propV} we have 
$$
({\mathbb V} P)_{\mathfrak p}=\operatorname{Hom}_{\overline{\mathcal O}_{\mathfrak p}} ({\overline{P}}^{\mathcal J}_{{\widetilde S}}(\lambda)\otimes_{\widetilde S}\widetilde S_{\mathfrak p},P_{\mathfrak p})
$$ 
and 
$$
({\mathbb V} M)_{\mathfrak p}=\operatorname{Hom}_{\overline{\mathcal O}_{\mathfrak p}} ({\overline{P}}^{\mathcal J}_{{\widetilde S}}(\lambda)\otimes_{\widetilde S}\widetilde S_{\mathfrak p},,M_{\mathfrak p}).
$$
We denote these spaces by ${\mathbb V}_{\mathfrak p} P_{\mathfrak p}$ and ${\mathbb V}_{\mathfrak p} M_{\mathfrak p}$, resp.
It is hence enough to show that 
$$
\operatorname{Hom}_{{\overline{\mathcal O}} _{\mathfrak p}}(P_{\mathfrak p},M_{\mathfrak p})\to\operatorname{Hom}_{{\mathcal{Z}}_{\mathfrak p}({{\mathcal K}})}({\mathbb V}_{\mathfrak p} P_{\mathfrak p},{\mathbb V}_{\mathfrak p} M_{\mathfrak p})
$$
 is an isomorphism for any prime ideal ${\mathfrak p}$ of height one.

If $\alpha^\vee\not\in{\mathfrak p}$ for all $\alpha\in R$, then 
$$
{\mathbb V}_{\mathfrak p}\cdot=\operatorname{Hom}_{{\overline{\mathcal O}} _{\mathfrak p}}({\overline{P}}^{{\mathcal J}}_{{\widetilde S}}(\lambda)\otimes_{\widetilde S} {{\widetilde S}}_{\mathfrak p},\cdot)\cong\bigoplus_{\mu\in{{\mathcal K}}}\operatorname{Hom}_{{\overline{\mathcal O}} _{\mathfrak p}}({\overline{\Delta}} _{\mathfrak p}(\mu),\cdot)
$$
and $M_{\mathfrak p}=\bigoplus_{\nu\in{{\mathcal K}}} (M_{\mathfrak p})_{[\mu]}$, where each $(M_{\mathfrak p})_{[\mu]}$ is isomorphic to a direct sum of copies of ${\overline{\Delta}} _{\mathfrak p}(\mu)$. As $\operatorname{End}_{{\overline{\mathcal O}} _{\mathfrak p}}({\overline{\Delta}} _{\mathfrak p}(\mu))$ and $\operatorname{End}_{{\mathcal{Z}}_{\mathfrak p}({{\mathcal K}})}({\mathbb V}_{\mathfrak p}{\overline{\Delta}} _{\mathfrak p}(\mu))$ are free over ${{\widetilde S}}_{\mathfrak p}$ of rank one generated by the respective identities, we get the claim.

Suppose now that  ${\mathfrak p}=\alpha^\vee R$, and set $A:=\{\mu\in {\mathcal K}\mid \alpha\uparrow\mu,\alpha\downarrow\mu\not\in{\mathcal K}\}$ and $B:=\{\mu\in{\mathcal K}\mid \alpha\uparrow\mu\in{\mathcal K}\}$. Then we have an isomorphism
$$
{\overline{P}}_{{\widetilde S}}^{\mathcal J}(\lambda)\otimes_{{\widetilde S}} {{\widetilde S}}_{\mathfrak p}\cong\bigoplus_{\mu\in A}{\overline{\Delta}} _{\mathfrak p}(\mu)\oplus\bigoplus_{\mu\in B} {\overline{P}}_{\mathfrak p}(\mu).
$$
and a corresponding decomposition $M_{\mathfrak p}=\bigoplus_{\mu\in A}(M_{\mathfrak p})_{[\mu]}\oplus \bigoplus_{\mu\in B} (M_{\mathfrak p})_{[\mu,\alpha\uparrow\mu]}$. For $\mu\in A$, the object $(M_{\mathfrak p})_{[\mu]}$ is isomorphic to a direct sum of copies of ${\overline{\Delta}} _{\mathfrak p}(\mu)$ and we argue as before. For $\mu\in B$ we have 
\begin{align*}
\operatorname{Hom}_{{\mathcal{Z}}_{\mathfrak p}({\mathcal K})}({\mathbb V}_{\mathfrak p}{\overline{P}}(\mu), {\mathbb V}_{\mathfrak p} N)&=\operatorname{Hom}_{{\mathcal{Z}}_{\mathfrak p}(\{\mu,\alpha\uparrow\mu\})}({\mathcal{Z}}_{\mathfrak p}(\{\mu,\alpha\uparrow\mu\}),{\mathbb V}_{\mathfrak p} N)\\
&={\mathbb V}_{\mathfrak p} N\\
&=\operatorname{Hom}_{{\overline{\mathcal O}} _{\mathfrak p}}(P_{\mathfrak p}(\mu),N).
\end{align*}

\end{proof}

\subsection{The Strukturfunktor on projectives}
Apart from the Endomorphismensatz and the Struktursatz, we need to understand the image of the projectives under ${\mathbb V}$.

\begin{Fiebigproposition}\label{prop-VPproj} For $\mu\in{{\mathcal K}}$ the object ${\mathbb V} {\overline{P}}^{\mathcal J}_{{\widetilde S}}(\mu)$ is projective in ${\mathcal C}_{{\widetilde S}}({{\mathcal K}})$.
\end{Fiebigproposition}
\begin{proof} Again let us abbreviate ${\overline{P}}_{{\widetilde S}}^{\mathcal J}(\mu)$ by $P$. In order to prove the proposition, we  show that ${\mathbb V} P$ satisfies the two assumptions listed in Proposition \ref{prop-projinC}.  We have to check the following.

{\bf Claim 1:} For any $\chi\in{{\mathcal K}}$, the ${\widetilde S}$-module $({\mathbb V} P)^\chi$ is free. 

Let $r$ be the multiplicity of ${\overline{\Delta}} _{{\widetilde S}}(\chi)$ in $P$. Hence $r=\dim_{\mathbb C}\operatorname{Hom}_{{\overline{\mathcal O}} _{\mathbb C}}(P_{\mathbb C},{\overline{\Delta}} _{\mathbb C}(\chi))$. Let $f_1^\prime,\dots,f_r^\prime$ be a ${\mathbb C}$-basis of this space and denote by $f_i\in\operatorname{Hom}_{{\overline{\mathcal O}} _{{\widetilde S}}}(P,{\overline{\Delta}} _{{\widetilde S}}(\chi))$ a lift of $f_i^\prime$. Then set $f=(f_1,\dots,f_r)^T\colon P\to{\overline{\Delta}} _{{\widetilde S}}(\chi)^{\oplus r}$. We consider now ${\mathbb V} f\colon{\mathbb V} P\to{\mathbb V}{\overline{\Delta}} _{{\widetilde S}}(\chi)^{\oplus r}$. As ${\mathbb V}{\overline{\Delta}} _{{\widetilde S}}(\chi)$ is supported on $\{\chi\}$, the map ${\mathbb V} f$ factors over a homomorphim $({\mathbb V} f)^\chi\colon ({\mathbb V} P)^\chi\to{\mathbb V}{\overline{\Delta}} _{{\widetilde S}}(\chi)^{\bigoplus r}$. We claim that $({\mathbb V} f)^\chi$  is an isomorphism. As ${\mathbb V}{\overline{\Delta}} _{{\widetilde S}}(\chi)$ is free as an ${\widetilde S}$-module, this then implies that $({\mathbb V} P)^\chi$ is free.

The injectivity of $({\mathbb V} f)^\chi$ follows from the fact that $({\mathbb V} P)^\chi$ and ${\mathbb V}{\overline{\Delta}} (\chi)$ are torsion free as  ${\widetilde S}$-modules, and $({\mathbb V} f)^\chi$ is an isomorphism after applying the base change functor $\cdot\otimes_{{\widetilde S}} Q$. We hence have to show that it is surjective. For this it suffices to show that ${\mathbb V} f$ is surjective. As ${\mathbb V}$ commutes with base change, it suffices to show that $({\mathbb V} f)_{\mathbb C}$ is surjective. Let us denote by $P_{\mathbb C}^\chi$ the image of $f_{\mathbb C}\colon P_{\mathbb C}\to{\overline{\Delta}} _{\mathbb C}(\chi)^{\oplus r}$. Then the following statements hold:
\begin{enumerate}
\item The natural homomorphism $\operatorname{Hom}_{{\overline{\mathcal O}} _{\mathbb C}}(P_{\mathbb C}^\chi,{\overline{\Delta}} _{\mathbb C}(\chi))\to \operatorname{Hom}_{{\overline{\mathcal O}} _{\mathbb C}}(P_{\mathbb C},{\overline{\Delta}} _{\mathbb C}(\chi))$ is an isomorphism.
\item We have ${\operatorname{im}}{\mathbb V}_{\mathbb C} f_{\mathbb C}={\mathbb V}_{\mathbb C} P_{\mathbb C}^\chi$.
\end{enumerate}
The first statement follows directly from the construction of $f$ and the definition of $P_{\mathbb C}^\chi$,  and the second is due to the projectivity of ${\overline{P}}_{\mathbb C}^{\mathcal J}(\lambda)$. Now we consider the commutative diagram

\centerline{
\xymatrix{
\operatorname{Hom}_{{\overline{\mathcal O}} _{\mathbb C}}(P_{\mathbb C}^\chi,{\overline{\Delta}} _{\mathbb C}(\chi))\ar[r]^a\ar[d]^{b}&\operatorname{Hom}_{{\mathcal{Z}}_{\mathbb C}({{\mathcal K}})}({\mathbb V}_{\mathbb C} P_{\mathbb C}^\chi,{\mathbb V}_{\mathbb C} {\overline{\Delta}} _{\mathbb C}(\chi))\ar[d] \\
\operatorname{Hom}_{{\overline{\mathcal O}} _{\mathbb C}}(P_{\mathbb C},{\overline{\Delta}} _{\mathbb C}(\chi))\ar[r]^c&\operatorname{Hom}_{{\mathcal{Z}}_{\mathbb C}({{\mathcal K}})}({\mathbb V}_{\mathbb C} P_{\mathbb C},{\mathbb V}_{\mathbb C} {\overline{\Delta}} _{\mathbb C}(\chi)).
}
}
\noindent
By Lemma \ref{lemma-antidom}, the homomorphism $c$ is injective. As we argued above,  $b$ is an isomorphism. Hence $a$ is injective and, as $\dim_{\mathbb C} \operatorname{Hom}_{{\overline{\mathcal O}} _{\mathbb C}}(P_{\mathbb C},{\overline{\Delta}} _{\mathbb C}(\chi))=r$, we deduce $\dim_{\mathbb C} 
\operatorname{Hom}_{{\mathcal{Z}}_{\mathbb C}({{\mathcal K}})}({\mathbb V}_{\mathbb C} P_{\mathbb C}^\chi,{\mathbb V}_{\mathbb C} {\overline{\Delta}} _{\mathbb C}(\chi))\ge r$. As ${\mathbb V}_{\mathbb C}{\overline{\Delta}} _{\mathbb C}(\chi)$ is one-dimensional as a vector space, this implies that $\dim_{\mathbb C} {\mathbb V}_{\mathbb C} P_{\mathbb C}^\chi\ge r$. Using the remark above, we obtain that ${\operatorname{im}} {\mathbb V}_{\mathbb C} f_{\mathbb C}\subset{\mathbb V}_{\mathbb C}{\overline{\Delta}} _{\mathbb C}(\chi)^{\oplus r}$ is a subvector space of dimension at least $r$, but as the vector space on the right has dimension $r$, we obtain 
${\operatorname{im}} {\mathbb V}_{\mathbb C} f_{\mathbb C}={\mathbb V}_{\mathbb C}{\overline{\Delta}} _{\mathbb C}(\chi)^{\oplus r}$, i.e.\ , $({\mathbb V} f)_{\mathbb C}={\mathbb V}_{\mathbb C} f_{\mathbb C}$ is surjective, hence $({\mathbb V} f)^\chi$ is surjective.

{\bf Claim 2:} Let $E\colon\chi\stackrel{\alpha^\vee}{\text{---\!\!\!---\!\!\!---}}\chi^\prime$ be an edge, and suppose that $\chi<\chi^\prime$. Then $({\mathbb V} P)^E_\chi=\alpha^\vee ({\mathbb V} P)^\chi$. 

Note first that $\chi^\prime=s_{\alpha+n\delta}\chi$ for some $n\in{\mathbb Z}$. As $\chi$ and $\chi^\prime$ are contained in the critical hyperplane $\Lambda$ we have that $s_{\alpha+m\delta}s_{\alpha+n\delta}\chi-\chi$ is a multiple of $\delta$. Our assumption on ${\mathcal K}$ hence implies that the $\langle s_{\alpha+m\delta}\mid m\in{\mathbb Z}\rangle$ orbit through $\chi$ intersected with ${{\mathcal K}}$ is $\{\chi,\chi^\prime\}$. The description of the subgeneric projectives in Proposition \ref{prop-strucproj} then shows that $P\otimes_{{\widetilde S}} {{\widetilde S}}_\alpha$ is a non-split extension of ${\overline{\Delta}} _{{{\widetilde S}}_\alpha}(\chi)$ and ${\overline{\Delta}} _{{{\widetilde S}}_\alpha}(\chi^\prime)$. Once this is established, we can argue as in the proof of Proposition 7.2 in \cite{FieAdv}.
\end{proof}

\section{The main result}\label{sec-mT}

As before we fix a $\sim_{{\widetilde S}}$-equivalence class $\Lambda$ in ${\widehat{\mathfrak h}^\star}_{{\operatorname{crit}}}$. 
Let  ${\mathcal J}$ be an open and locally bounded subset of $\Lambda$ and fix $\lambda\in{\mathcal J}$. We set ${\mathcal K}={\mathcal J}\cap\{\ge\lambda\}$. We assume that our data satisfies the properties mentioned in the introduction to Section \ref{sec-ST}. We now consider the subquotient category ${\widetilde{\mathcal O}} _{{\widetilde S}}:={\widetilde{\mathcal O}} _{{\widetilde S},\Lambda}^{[{\mathcal K}]}$, and in particular its full subcategory ${\widetilde{\mathcal O}} _{{\widetilde S}}^{V}$ of objects that admit a Verma flag.  We define ${\widetilde{\mathcal Z}}_{{\widetilde S}}:={\mathcal{Z}}_{{\widetilde S}}({\mathcal K})$, and we denote by ${\widetilde{\mathcal C}}_{{\widetilde S}}$ the subcategory of ${\widetilde{\mathcal Z}}_{{\widetilde S}}{\operatorname{-mod}}$ that contains all objects that admit a Verma flag. For $\mu\in{{\mathcal K}}$ we set
 ${\widetilde{P}} _{{\widetilde S}}(\mu):=T {\overline{P}}^{{\mathcal J}}_{{\widetilde S}}(\mu)$.  Note that by Lemma \ref{lemma-projinquot} we have $\operatorname{End}_{{\widetilde{\mathcal O}} _{{\widetilde S}}}({\widetilde{P}} (\lambda))=\operatorname{End}_{{\overline{\mathcal O}} _{{\widetilde S}}}({\overline{P}}_{{\widetilde S}}^{\mathcal J}(\lambda))={\widetilde{\mathcal Z}}_{{\widetilde S}}$.
 Hence we can consider the Strukturfunktor 
$$
{\widetilde{\mathbb V}}:=\operatorname{Hom}_{{\widetilde{\mathcal O}}_{{\widetilde S}}}({\widetilde{P}}_{{\widetilde S}}(\lambda),\cdot)\colon {\widetilde{\mathcal O}} _{{\widetilde S}}\to{\widetilde{\mathcal Z}}_{{\widetilde S}}{\operatorname{-mod}}
$$
for the quotient category. By Lemma \ref{lemma-projinquot} we have an isomorphism
$$
{\widetilde{\mathbb V}}\circ T={\mathbb V}\colon {\overline{\mathcal O}} _{{\widetilde S}}\to{\widetilde{\mathcal Z}}_{{\widetilde S}}{\operatorname{-mod}}.
$$
\begin{Fiebiglemma} \begin{enumerate}
\item  If $M$ is an object in ${\widetilde{\mathcal O}} _{{\widetilde S}}$ that admits a Verma flag, then ${\widetilde{\mathbb V}} M$ admits a Verma flag as well.
\item The induced functor ${\widetilde{\mathbb V}}\colon{\widetilde{\mathcal O}} _{{\widetilde S}}^{V}\to{\widetilde{\mathcal C}}_{{\widetilde S}}$ is exact (with respect to the non-standard exact structure on ${\widetilde{\mathcal C}}_{{\widetilde S}}$ that we defined in Section \ref{sec-defC}).
\end{enumerate}
\end{Fiebiglemma} 
\begin{proof}  We have ${\widetilde{\mathbb V}}({\widetilde{\Delta}} _{{\widetilde S}}(\mu))={\widetilde{\mathbb V}}(T{\overline{\Delta}} _{{\widetilde S}}(\mu))={\mathbb V}({\overline{\Delta}} _{{\widetilde S}}(\mu))\cong{\mathcal V}_{{\widetilde S}}(\mu)$. As ${\widetilde{\mathbb V}}$ is an exact functor (with respect to the standard exact structure on ${\widetilde{\mathcal Z}}_{{\widetilde S}}{\operatorname{-mod}}^f$), we deduce $\operatorname{supp}_{{\widetilde{\Delta}} } M=\operatorname{supp}_{{\mathcal C}} {\widetilde{\mathbb V}}(M)$ and hence ${\widetilde{\mathbb V}}(M^{{\mathcal J}^\prime})=({\widetilde{\mathbb V}}(M))^{{\mathcal J}^\prime}$ for any open subset ${\mathcal J}^\prime$ of $\Lambda$. As $M^{{\mathcal J}^\prime}$ admits a Verma flag as well, ${\widetilde{\mathbb V}}(M)^{{\mathcal J}^\prime}$ is free as an ${\widetilde S}$-module. The last argument also proves that ${\widetilde{\mathbb V}}$ is exact even for the non-standard exact structure on ${\widetilde{\mathcal C}}_{{\widetilde S}}$.
\end{proof}

Now we consider ${\widetilde{P}} :=\bigoplus_{\mu\in{{\mathcal K}}}{\widetilde{P}} (\mu)$ and we set ${\mathbf{A}}:=\operatorname{End}_{{\widetilde{\mathcal O}} _{{\widetilde S}}}({\widetilde{P}} )$. Then the functor $\operatorname{Hom}_{{\widetilde{\mathcal O}} _{{\widetilde S}}}({\widetilde{P}} ,\cdot)\colon{\widetilde{\mathcal O}} _{{\widetilde S}}\to{\operatorname{mod-}}{\mathbf{A}}$ is an equivalence of categories and we denote by ${\mathcal{A}}\subset{\operatorname{mod-}}{\mathbf{A}}$ the image of the subcategory ${\widetilde{\mathcal O}} _{{\widetilde S}}^{V}$.   
The functorial homomorphism ${\mathbf{A}}=\operatorname{End}_{{\widetilde{\mathcal O}} }({\widetilde{P}} _{{\widetilde S}})\to\operatorname{End}_{{\widetilde{\mathcal C}}_{{\widetilde S}}}({\widetilde{\mathbb V}}_{{\widetilde S}} {\widetilde{P}} _{{\widetilde S}})$ is an isomorphism, as it identifies with the functorial homomorphism $\operatorname{End}_{{\overline{\mathcal O}} _{{\widetilde S}}}(\bigoplus_{\mu\in{{\mathcal K}}}{\overline{P}}_{{\widetilde S}}^{\mathcal J}(\mu))\to \operatorname{End}_{{\widetilde{\mathcal C}}_{{\widetilde S}}}({\mathbb V} \bigoplus_{\mu\in{{\mathcal K}}}{\overline{P}}_{{\widetilde S}}^{\mathcal J}(\mu))$ (by Lemma \ref{lemma-projinquot}), and by the Struktursatz the latter homomorphism is an isomorphism. So from now on we identify $\operatorname{End}_{{\widetilde{\mathcal Z}}_{{\widetilde S}}}({\mathbb V} {\widetilde{P}} _{{\widetilde S}})$ with ${\mathbf{A}}$ via this isomorphism. In particular, we obtain a functor  \begin{align*}
{\mathbb Y}\colon{\widetilde{\mathcal C}}_{{\widetilde S}}&\to{\operatorname{mod-}}{\mathbf{A}},\\
M&\mapsto\operatorname{Hom}_{{\widetilde{\mathcal C}}_{{\widetilde S}}}({\widetilde{\mathbb V}} {\widetilde{P}} , M).
\end{align*}

\begin{Fiebiglemma} The image of the functor ${\mathbb Y}$ is contained in ${\mathcal{A}}\subset{\operatorname{mod-}}{\mathbf{A}}$. 
\end{Fiebiglemma} 
\begin{proof} By Proposition \ref{prop-VPproj}, ${\widetilde{\mathbb V}} {\widetilde{P}} $ is a projective object in ${\widetilde{\mathcal C}}_{{\widetilde S}}$. Hence the functor ${\mathbb Y}\colon{\widetilde{\mathcal C}}_{{\widetilde S}}\to{\operatorname{mod-}}{\mathbf{A}}$ is exact (we consider the natural exact structure on ${\operatorname{mod-}}{\mathbf{A}}$). Hence we only have to show that ${\mathbb Y}{\mathcal V}_{{\widetilde S}}(\mu)$ is a Verma object in ${\operatorname{mod-}}{\mathbf{A}}$. This follows from the fact that ${\mathcal V}_{{\widetilde S}}(\mu)={\mathbb V}{\overline{\Delta}} _{{\widetilde S}}(\mu)$ and hence 
\begin{align*}
{\mathbb Y}{\mathcal V}_{{\widetilde S}}(\mu)&=\operatorname{Hom}_{{\widetilde{\mathcal C}}_{{\widetilde S}}}({\widetilde{\mathbb V}} {\widetilde{P}} ,{\mathbb V}{\overline{\Delta}} _{{\widetilde S}}(\mu))\\
&=\operatorname{Hom}_{{\widetilde{\mathcal O}} _{{\widetilde S}}}({\widetilde{P}} ,{\widetilde{\Delta}} _{{\widetilde S}}(\mu)),
\end{align*}
where we used Theorem \ref{thm-ff}.
\end{proof}

\begin{Fiebiglemma}   \label{lemma-comdiag} The diagram of functors

\centerline{
\xymatrix{
{\widetilde{\mathcal O}} ^{V}_{{\widetilde S}}\ar[dd]_{{\widetilde{\mathbb V}}}\ar[rrd]^{{\mathbb X}}&& \\
&&{\mathcal{A}} \\
{\mathcal C}_{{\widetilde S}}\ar[rru]^{{\mathbb Y}}&& 
}
}
\noindent
commutes naturally.
\end{Fiebiglemma} 
\begin{proof}  The diagram commutes naturally as, by Theorem \ref{thm-ff}, the  functorial map ${\mathbb X}(M)= \operatorname{Hom}_{{\widetilde{\mathcal O}} _{{\widetilde S}}}({\widetilde{P}} ,M)\to\operatorname{Hom}_{{\widetilde{\mathcal C}}_{{\widetilde S}}}({\widetilde{\mathbb V}} {\widetilde{P}} ,{\widetilde{\mathbb V}} M)={\mathbb Y}({\widetilde{\mathbb V}} M)$ is an isomorphism.  
\end{proof}

\begin{Fiebiglemma} \label{lemma-Yfe}
The functor ${\mathbb Y}\colon {\widetilde{\mathcal C}}_{{\widetilde S}}\to{\mathcal{A}}$ is a faithful embedding, i.e.\  it is injective on isomorphism classes and on $\operatorname{Hom}$-spaces.
\end{Fiebiglemma}

\begin{proof} Let $e_\lambda\in\operatorname{End}_{{\widetilde{\mathcal C}}_{{\widetilde S}}}(\bigoplus_{\mu\in{\mathcal K}}{\widetilde{\mathbb V}}{\widetilde{P}} (\mu))$ be the idempotent corresponding to the direct summand ${\widetilde{\mathbb V}} {\widetilde{P}} (\lambda)$. Then 
\begin{align*}
({\mathbb Y} M)e_\lambda &= \operatorname{Hom}_{{\widetilde{\mathcal Z}}_{{\widetilde S}}}({\widetilde{\mathbb V}} {\widetilde{P}} (\lambda), M) \\
&\cong \operatorname{Hom}_{{\widetilde{\mathcal Z}}_{{\widetilde S}}}({\widetilde{\mathcal Z}}_{{\widetilde S}}, M) \\
&=M
\end{align*} 
as a ${\widetilde{\mathcal Z}}_{{\widetilde S}}$-module.
\end{proof}

Here is our main result.

\begin{Fiebigtheorem}  The functor ${\widetilde{\mathbb V}}\colon {\widetilde{\mathcal O}} ^V_{{\widetilde S}}\to{\widetilde{\mathcal C}}_{{\widetilde S}}$ is an equivalence of categories.
\end{Fiebigtheorem} 
\begin{proof} Consider the commutative diagram of functors that we established in Lemma \ref{lemma-comdiag}. 
Note that the functor ${\mathbb X}$ is an equivalence (by  definition of ${\mathcal{A}}$) and ${\mathbb Y}$ is a faithful embedding by Lemma \ref{lemma-Yfe}. Hence ${\widetilde{\mathbb V}}$ is an equivalence of categories.

\end{proof}


\begin{thebibliography}{GKM98}
\bibitem[AF1]{AF12} Tomoyuki Arakawa, Peter Fiebig, {\em On the restricted Verma modules at the critical level},  Trans. Amer. Math. Soc. {\bf 364} (2012), 4683-4712. 
\bibitem[AF2]{AFLin} Tomoyuki Arakawa, Peter Fiebig, {\em The linkage principle for restricted critical level representations of affine Kac-Moody algebras}, Compos. Math. {\bf 148} (2012), 1787-1810. 
\bibitem[F1]{FieMZ} Peter Fiebig, {\em Centers and translation functors for the category ${\mathcal{O}}$ over Kac-Moody algebras}, Math.~ Z., {\bf 243} (2003), no.~ 4, 689--717.
\
\bibitem[F2]{FieAdv} Peter Fiebig, \emph{Sheaves on moment graphs and a localization of {V}erma  flags}, {A}dv. {M}ath. \textbf{217} (2008), 683--712.
\bibitem[F3]{FProj} Peter Fiebig, \emph{On the restricted projective objects in the affine category O at the critical level}, in {\em Algebraic Groups and Quantum Groups} (Nagoya, Japan, 2010),  Contemp. Math. {\bf 565} (2012), 55-70.
\bibitem[F4]{Fsub} Peter Fiebig, \emph{On the subgeneric restricted blocks of affine category O at the critical level} in {\em Symmetries, Integrable Systems and Representations}, Springer Proceedings in Mathematics and Statistics {\bf 40} (2013), 65-84.
\bibitem[G]{Gab} Pierre Gabriel, {\em Des cat\'egories ab\'eliennes}, Bull. Soc. Math. France {\bf 90} (1962), 323--448.
\bibitem[K]{K} Johannes K\"ubel, {\em Centers for the restricted category ${\mathcal{O}}$ at the critical level over affine Kac--Moody algebras}, Math. Z. {\bf 276} (2014), 1133--1149.
\bibitem[MP]{MP} 
Robert V.~Moody, Arturo Pianzola, {\em Lie algebras with triangular decompositions},
Canadian Mathematical Society Series of Monographs and Advanced Texts. John Wiley \& Sons, New York (1995).
\bibitem[N]{N} 
Wayne Neidhardt, {\em Verma module imbeddings and the Bruhat order for Kac-Moody algebras}, J.~Algebra {\bf 109} (1987), 430--438.
\bibitem[S1]{SJAMS}
Wolfgang Soergel, {\em Kategorie ${\mathcal{O}}$, perverse Garben und Moduln {\"u}ber den Koinvarianten zur Weylgruppe}, J. Am. Math. Soc. {\bf 3} (1990), No.2, 421--445.
\bibitem[S2]{S} Wolfgang Soergel, \emph{On the relation between intersection cohomology and representation theory in positive characteristic}, Journal of Pure and Applied Algebra {\bf 152} (2000), 311-335.
\bibitem[W]{W} Geordie Williamson, {\em Schubert calculus and torsion}, preprint 2013, {\tt arxiv:1309.5055}.
\end{thebibliography}
\end{document}